\documentclass[12pt]{article}

\usepackage{graphicx}
\usepackage{geometry}
\usepackage{amsmath,amssymb,amsthm}
\usepackage{booktabs}
\usepackage[labelsep=period]{caption}
\usepackage{graphicx,booktabs,bm,enumitem}
\usepackage{cases}
\usepackage{url}
\usepackage[longnamesfirst]{natbib}

\usepackage[T1]{fontenc}
\usepackage{newtxtext,newtxmath}

\newcommand{\footremember}[2]{
	\footnote{#2}
	\newcounter{#1}
	\setcounter{#1}{\value{footnote}}\
}
\newcommand{\footrecall}[1]{
	\footnotemark[\value{#1}]\
} 

\newcommand{\bb}[1]{\boldsymbol{#1}}

\DeclareMathOperator{\Var}{Var}
\DeclareMathOperator{\Expec}{E}
\DeclareMathOperator{\Prob}{P}

\newtheorem{thm}{Theorem}
\newtheorem{lem}{Lemma}

\title{Tests for validity of the semiparametric heteroskedastic transformation model}
\author{Marie Hu\v{s}kov\'a\footremember{uk}{Department of Probability and Mathematical Statistics, Charles University, Prague, Czech Republic}
	\and Simos G. Meintanis\footremember{nwu}{Unit for Business Mathematics and Informatics, North-West University, Potchefstroom, South~Africa}\footnote{Corresponding author: \texttt{simosmei@econ.uoa.gr}}
	\and Charl Pretorius\footrecall{uk}\footrecall{nwu}}

\begin{document}
\sloppy

\begin{titlepage}
\maketitle

\begin{abstract}
	There exist a number of tests for assessing the nonparametric heteroscedastic location-scale assumption. Here we consider a goodness-of-fit test for the more general hypothesis of the validity of this model under a parametric functional  transformation on the response variable. Specifically we consider testing for independence between the regressors and the errors in a model where the transformed response is just a location/scale shift of the error. Our criteria use the familiar factorization property of the joint characteristic function of the covariates under independence. The difficulty is that the errors are unobserved and hence one needs to employ properly estimated residuals in their place.  We study the limit distribution of the test statistics under the null hypothesis as well as under alternatives, and  also suggest  a resampling procedure in order to approximate the critical values of the tests. This resampling is subsequently employed in a series of Monte Carlo experiments that illustrate the finite-sample properties of the new test. We also investigate the performance of related test statistics for normality and symmetry of errors, and apply our methods on real data sets.
	\vspace{.1in}
	\vspace{.1in}
\end{abstract}

\end{titlepage}

\section{Introduction}
At least since the seminal paper of \citet{box1964} transformations are applied to data sets in order to facilitate statistical inference. The goal of a certain transformation could be, among others, reduction of skewness, faster convergence to normality and better fit, linearity, and stabilization of variance. These aims, which may even be contradictory, are quite important as it is only under such assumptions that certain statistical procedures are applicable. Some of the issues raised when performing a certain transformation are discussed in varying settings by \citet{staniswalis1993}, \citet{quiroz1996}, \citet{yeo2000}, \citet{chen2002}, \citet{mu2007}, and \citet{meintanis2015}, and in the reviews of \citet{sakia1992} and \citet{horowitz2009}. Here we will consider goodness-of-fit (GOF)  tests for the (after-transformation) location-scale nonparametric heteroskedastic model
\begin{equation}\label{modelI}
	\mathcal{T}(Y) = m(\boldsymbol X) + \sigma(\boldsymbol X) \varepsilon,
\end{equation} where $\mathcal{T}(\cdot)$ is a transformation acting on the response $Y$, $m(\cdot)$ and $\sigma(\cdot)$ are unknown  functions, and where the error $\varepsilon$, having mean zero and unit variance, is supposed to be independent of the vector of covariates $\boldsymbol X$. The classical location-scale model, i.e.\ the model in \eqref{modelI} with $\mathcal{T}(Y)\equiv Y$, is a popular model that is often employed in statistics as well as in econometrics (see e.g., \citealp{racine2017a,brown2007,chen2005}), and there exist a number of approaches to test the validity of this model such as the classical Kolmogorov--Smirnov and Cram\'er--von Mises tests in  \citet{einmahl2008} and the test criteria in \citet{hlavka2011} which are based on the characteristic function. On the other hand the problem of goodness-of-fit (GOF) for the general model \eqref{modelI} under any given (fixed) transformation has only recently drawn attention in \citet{neumeyer2016} by means of classical methods. Here we deviate from classical approaches in that we employ the characteristic function (CF) instead of the distribution function (DF) as the basic inferential tool. As already mentioned the CF approach was also followed in analogous situations  by \citet{SZ2005}, \citet{hlavka2011}, and  \citet{huskova2018}, among others, and gave favorable results. 

The rest of the paper is outlined as follows.  In Section 2 we introduce the null hypothesis of independence between the regressor and the error term and formulate the new test statistic. The asymptotic distribution of the test statistic under the null hypothesis as well as under alternatives is studied in Section 3, while in Section 4 we particularize our method, and suggest a bootstrap procedure for its calibration. Section 5 presents the results of a Monte Carlo study. Since one aim of transforming the response is to achieve normality, or more generally symmetry after-transformation, we also investigate the small-sample performance of CF-based statistics for these problems with reference to the regression errors. Real data applications are also included. We finally conclude with discussion of our findings in  Section 7. Some technical material is deferred to the Appendix.

\section{Null hypothesis and test statistics}

Note that underlying equation  \eqref{modelI} is the potentiality of obtaining a location/scale structure following a certain transformation of the response. Hence there exist a number of inferential problems similar to the problems faced in the non-transformation case, i.e., when  $\mathcal{T}(Y)\equiv Y$. For the homoskedastic version of model \eqref{modelI} with $\sigma \equiv$ constant, estimation methods were proposed by \citet{linton2008} and \citet{colling2015}. The classical problem of fitting a specific regression function was considered by \citet{colling2016,colling2017}, respectively, by means of the DF and the integrated regression function, while regressor significance using Bieren's CF-type approach  was considered by \citet{allison2018}. On the other hand, the problem of GOF of the model itself is studied by \citet{huskova2018} only in the homoskedastic case, whereas the single existing work for GOF testing with the more general heteroskedastic model is that of \citet{neumeyer2016} which is based on the classical pair of Kolmogorov--Smirnov/Cram\'er--von Mises functionals.

Here we are concerned with the GOF test for a fixed parametric transformation ${\cal{Y}}_{\boldsymbol \vartheta}=\mathcal{T}_{\boldsymbol \vartheta}(Y)$,  indexed by a parameter ${\bb{\vartheta}} \in \Theta$. Specifically on the basis of  independent copies $(Y_j,\boldsymbol X_j),\,j=1,\ldots, n$, of $(Y,\boldsymbol X)\in \mathbb R\times \mathbb R^p$ we wish to test the null hypothesis
\begin{equation}\label{null}
{\cal{H}}_0 : \exists \, \bb{\vartheta}_0 \in \Theta \text{\, such that \,}
\varepsilon_{\bb{\vartheta}_0}(Y,\boldsymbol X) \bot \boldsymbol X,
\end{equation}
where  $\bot$ denotes stochastic independence, and
\begin{equation}\label{error}
\varepsilon_{\bb{\vartheta}}(Y,\boldsymbol X)=\frac{{\cal{Y}}_{\boldsymbol \vartheta}-m_{\bb{\vartheta}}(\boldsymbol X)}{\sigma_{\bb{\vartheta}}(\boldsymbol X)},
\end{equation}
with $m_{\bb{\vartheta}}(\boldsymbol X):=\mathbb E({\cal{Y}}_{\boldsymbol \vartheta}|\boldsymbol X)$ and $\sigma^2_{\bb{\vartheta}}(\boldsymbol X):={\rm{Var}}({\cal{Y}}_{\boldsymbol \vartheta}|\boldsymbol X)$ being the mean and variance, respectively, of the transformed response conditionally on the covariate vector $\boldsymbol X$, and $\Theta\subseteq \mathbb R^q, \ q\geq 1$. Note that the validity of the null hypothesis ${\cal{H}}_0$  is tantamount to the existence of a (true) value $\boldsymbol \vartheta_0$ of $\boldsymbol \vartheta$ which if substituted in this specific parametric transformation
and applied on the response will render the location/scale structure of model \eqref{modelI}. To avoid confusion we emphasize that while the transformation is indeed parametric, the regression and the heteroskedasticity functions, respectively $m_{\bb{\vartheta}}(\cdot)$ and $\sigma_{\bb{\vartheta}}(\cdot)$, and apart from their implicit dependence on the particular transformation $\mathcal{T}_{\boldsymbol \vartheta}(\cdot)$ and the associated parameter $\bb\vartheta$, they are both viewed and estimated within an entirely nonparametric context. Therefore  our model can be labelled as a semiparametric model, i.e. parametric in the transformation but nonparametric in its location and scale functions.


To motivate our test statistic write $\varphi_{\boldsymbol X,\varepsilon_{\bb{\vartheta}}}$ for the joint CF of $(\boldsymbol X,\varepsilon_{\bb{\vartheta}})$, and $\varphi_{\boldsymbol X}$ and $\varphi_{\varepsilon_{\bb{\vartheta}}}$ for the marginal CFs of $\boldsymbol X$ and $\varepsilon_{\bb{\vartheta}}$, respectively, and recall that the null hypothesis in ({\ref{null})  equivalently implies that $\varphi_{\boldsymbol X,\varepsilon_{\bb{\vartheta}_0}}=\varphi_{\boldsymbol X}\varphi_{\varepsilon_{\bb{\vartheta}_0}}$, so that by following the approach in \citet{huskova2018}, the suggested test procedure will be based on the criterion
\begin{equation} \label{testst}
\Delta_{n,W}= n \int_{-\infty}^\infty \int_{\mathbb R^p} |\widehat \varphi(t_1,\boldsymbol t_2)-\widehat \varphi_{\boldsymbol X}(\boldsymbol t_2)\widehat
\varphi_{\widehat \varepsilon}(t_1)|^2
W(t_1,\boldsymbol t_2) dt_1d\boldsymbol t_2,
\end{equation}
where
$$
	\widehat \varphi(t_1,\boldsymbol t_2)=\frac{1}{n} \sum_{j=1}^n \exp\{i \boldsymbol t_2^\top \boldsymbol X_j+it_1 \widehat \varepsilon_j\}
 $$
is the joint empirical CF, and  $\widehat \varphi_{ \boldsymbol X}({\boldsymbol{t_2}})$ and $\widehat \varphi_{\widehat \varepsilon}(t_1)$ are the empirical marginal CFs resulting from $\widehat \varphi(t_1,\boldsymbol t_2)$ by setting $t_1=0$ resp. ${\boldsymbol {t}}_2=\boldsymbol 0$. These three quantities serve as estimators of $\varphi_{\boldsymbol X,\varepsilon_{\bb{\vartheta}_0}}$, $\varphi_{ \boldsymbol X}$ and $\varphi_{\varepsilon_{{\bb{\vartheta}}_0}}$, respectively, and they will be computed by means of   properly estimated residuals $\widehat \varepsilon_j=({\widehat {\cal{Y}}}_j-\widehat m(\boldsymbol X_j))/\widehat \sigma(\boldsymbol X_j), \ j=1,\ldots,n$. We note that in these residuals the observed response is parametrically transformed by means of ${\widehat {\cal{Y}}}_j=\mathcal{T}_{\widehat{\boldsymbol{\vartheta}}}(\boldsymbol X_j)$, using the particular transformation under test and the corresponding estimate $\widehat {\boldsymbol \vartheta}$ of the transformation parameter $\boldsymbol \vartheta$. Other than that  and as already mentioned, the estimate $\widehat m(\cdot)$ of the regression function as well as the heteroskedasticity estimate $\widehat \sigma(\cdot)$ are obtained entirely nonparametrically. Having said this, we often suppress the index $\widehat {\boldsymbol \vartheta}$ in these estimates. Clearly for any weight function $W$ satisfying $W(\cdot,\cdot)\geq 0$, the test statistic $\Delta_{n,W}$ defined in \eqref{testst} is expected to be large under alternatives, and therefore large values indicate that the null hypothesis is violated.

\section{Theoretical results}
We now consider theoretical properties of the introduced test statistics. More precisely, we  present the limit distribution of our test statistics under both the null as well as alternative hypotheses. Since the assumptions are quite technical they are deferred to the Appendix.

We first introduce some required notation.
For $\bb{\vartheta}\in \Theta$, define
\begin{equation*}
	m_{\bb{\vartheta}} (\boldsymbol X_{j})= \Expec\big({\mathcal{T}}_{\bb{\vartheta}} (Y_j)| \boldsymbol X_{j}\big)
	\qquad\text{and}\qquad
	\sigma^2_{\bb{\vartheta}} (\boldsymbol X_{j})= \Var \big({\mathcal{T}}_{\bb{\vartheta}} (Y_j)| \boldsymbol X_{j}\big).
\end{equation*}
Also define kernel estimators of $ m_{\bb {\vartheta}}(\bb{x})$ and $\sigma^2_{\bb {\vartheta}}(\bb{x})$, $\bb{x}=(x_1,\ldots, x_{p})^\top$, by
  \begin{align*}
    \widehat m_{\bb {\vartheta}}(\bb{x})&=\frac{1}{\widehat f(\bb{x})}
    \frac{1}{nh^{p}} \sum_{v=1}^n K\Big(\frac{ \bb{x}-\bb{X}_{v}}{h}\Big) {\mathcal{T}}_{\bb {\vartheta}}(Y_j),\\
     \widehat \sigma^2_{\bb {\vartheta}}(\bb{x})&=\frac{1}{\widehat f(\bb{x})}
    \frac{1}{nh^{p}} \sum_{v=1}^n  K\Big(\frac{ \bb{x}-\bb{X}_{v}}{h}\Big)\Big({\mathcal{T}}_{\bb {\vartheta}}(Y_j)-\widehat m_{\bb {\vartheta}}(\bb{X_j})\Big)^2,
    \end{align*}
respectively, where $K(\cdot)$ and $h=h_n$ are a kernel and a bandwidth, and
\[
	\widehat f(\bb x)=\frac{1}{nh^p}\sum_{v=1}^n K\Big(\frac{ \bb{x}-\bb{X}_{v}}{h}\Big)
\]
is a kernel estimator of the density of $\bb X_j$. Finally, let
\begin{equation}\begin{split}\label{residual}
	\varepsilon_{\bb{\vartheta},j}=\frac{ {\mathcal{T}}_{\bb {\vartheta}}(Y_j)-m_{\bb {\vartheta}}(\bb{X}_{j})}{\sigma_{\bb{\vartheta}} (\boldsymbol X_{j})},
	\qquad
	\varepsilon_{j}=\varepsilon_{\bb{\vartheta}_0,j},
	\qquad
  \widehat \varepsilon_j= \widehat \varepsilon_{\widehat{ \bb{\vartheta}},j}=\frac{ {\mathcal{T}}_{\widehat{\bb {\vartheta}}}(Y_j)-\widehat m_{\widehat{\bb {\vartheta}}}(\bb{X}_{j})}{\widehat \sigma_{\widehat{\bb{\vartheta}}} (\boldsymbol X_{j})},
\end{split}\end{equation}
where $\widehat{\bb {\vartheta}}$ is a $\sqrt n$-consistent estimator of $\bb {\vartheta}_0$.
It is assumed that  $\widehat{\bb {\vartheta}}$ allows an asymptotic representation as shown in assumption (A.7).

Now we formulate  the limit distribution of the test statistic under the null hypothesis:

\begin{thm}\label{thm1}
Let  assumptions (A.1)--(A.8) be
satisfied. Then  under the null hypothesis, as $n\to \infty$,
\[
\Delta_{n,W} \stackrel{d}{\to}
 \int_{\mathbb R^{p+1}} | Z(t_1,\boldsymbol t_2)|^2
W(t_1,\boldsymbol t_2) dt_1d\boldsymbol t_2,
\]
 where    $\{Z(t_1,\boldsymbol t_2),\, t \in \mathbb R^{p+1}\}$ is a Gaussian
process  with  zero mean function and the same covariance structure as
the process $\{Z_0(t_1,\boldsymbol t_2),\,( t_1,\boldsymbol t_2) \in \mathbb R^{p+1}\}$
defined as
\begin{equation}\begin{split}\label{process1}
Z_0(t_1, \boldsymbol t_2)={}& \{\cos(t_1 \varepsilon_1) - C_{\varepsilon}(t_2)\} g_+(\boldsymbol t_2^\top \boldsymbol X_1)+
 \{\sin(t_1 \varepsilon_1) - S_{\varepsilon}(t_1)\} g_-(\boldsymbol t_2^\top \boldsymbol X_1)
\\
&{}+ t_1\varepsilon_1\Big(S_{\varepsilon}(t_1) g_+(\boldsymbol t_2^\top \boldsymbol X_1)
 +C_{\varepsilon}(t_1) g_-(\boldsymbol t_2^\top \boldsymbol X_1)\Big)
 \\
&{}+\frac{1}{2}t_1(\varepsilon_1^2-1)\Big( C_{\varepsilon}'(t_1) g_+(\boldsymbol t_2^\top \boldsymbol X_1)-  S_{\varepsilon}'(t_1) g_-(\boldsymbol t_2^\top \boldsymbol X_1)\Big)
\\
&{}+ \boldsymbol g^\top(Y_1,\boldsymbol X_1)  \boldsymbol H_{\vartheta_0,q} (t_1,\boldsymbol t_2),
\end{split}\end{equation}
where $C_{\varepsilon}$ and $S_{\varepsilon}$ are the real and the imaginary part of the CF of $\varepsilon_1$.
Similarly, $C_X$ and $S_X$ denote the real and
the imaginary part of the CF of $\boldsymbol X_j$.  Also, $\boldsymbol g (Y_1,\boldsymbol X_1)$ is specified in  Assumption (A.7) and
\begin{equation*}
	\begin{split}
		g_+(\boldsymbol t_2^\top \boldsymbol X_1) ={}& \cos(\boldsymbol t_2^\top \boldsymbol X_1) +\sin (\boldsymbol t_2^\top \boldsymbol X_1)-C_X(\boldsymbol t_2)-S_X(\boldsymbol t_2),\\
		g_-(\boldsymbol t_2^\top \boldsymbol X_1) ={}& \cos(\boldsymbol t_2^\top \boldsymbol X_1) -\sin (\boldsymbol t_2^\top \boldsymbol X_1)-C_X(\boldsymbol t_2)+S_X(\boldsymbol t_2),\\
		\boldsymbol H_{\boldsymbol \vartheta,q} (t_1,\boldsymbol t_2) ={}&
		E\Big[\Big(\frac{\partial \mathcal{T}_{\boldsymbol \vartheta} (Y_1)}{\partial\vartheta_1},\ldots ,
		\frac{\partial \mathcal{T}_{\boldsymbol \vartheta} (Y_1)}{\partial\vartheta_q}   \Big)^\top\\
&{}\times  \Big\{\frac{1}{\sigma (\boldsymbol  X_1)}\Big(-t_1\sin(t_1\varepsilon_1)g_+(\boldsymbol t_2^\top\boldsymbol X_1)
 +t_1\cos(t_1\varepsilon_1)g_-(\boldsymbol t_2^\top\boldsymbol X_1)\Big)\\
&{} + \frac{1}{\sigma (\boldsymbol  X_2)} (1+\varepsilon_1 \varepsilon_2)\Big(t_1\sin(t_1\varepsilon_2)g_+(\boldsymbol t_2^\top\boldsymbol X_2)
 -t_1\cos(t_1\varepsilon_2)g_-(\boldsymbol t_2^\top\boldsymbol X_2)\Big)\Big\}\Big].
 \end{split}
\end{equation*}
\end{thm}
\noindent The proof is postponed to the Appendix.

The limit distribution under the null hypothesis is a weighted $L_2$-type functional of a Gaussian process. Concerning the structure of $Z_0(\cdot, \cdot)$, the first row in \eqref{process1} corresponds to the situation when  $\boldsymbol \vartheta_0$, $\varepsilon$, $m(\cdot)$, $\sigma^2(\cdot)$ are known, the second row reflects the influence of the estimator of $m(\cdot)$, the third one of the estimator of $\sigma^2(\cdot)$ while the last row reflects the influence of the estimator of $\boldsymbol{\vartheta}$.

To get an approximation of the critical value one estimates the unknown quantities and simulates the  limit distribution described above with unknown parameters replaced by their estimators. However, the bootstrap described in Section~\ref{sec_32} is probably more useful.

Concerning the consistency of the newly proposed test, note that
if $H_0$ is not true, there is no  parameter in $\bm\vartheta$ that leads to independence,
i.e.
$$
\forall \boldsymbol \vartheta,\quad  \varphi_{\boldsymbol X,\varepsilon_{\bb{\vartheta}}}\neq\varphi_{\boldsymbol X}\varphi_{\varepsilon_{\bb{\vartheta}}}.
$$
The main assertion under alternatives reads as follows.

\begin{thm}\label{thm2} The estimator $\bm{\widehat{\vartheta}}$ converges in probability to some ${\bm\vartheta}^0\in \Theta$  and let  $\bb{\vartheta}^0\in\Theta$  satisfy
\begin{equation}\label{alt}
 \int_{ \mathbb R^p} | \varphi_{\boldsymbol X,\varepsilon_{\bb{\vartheta}^0}} (\boldsymbol t_2,t_1)-\varphi_{\boldsymbol X} (\boldsymbol t_2)  \varphi_{\varepsilon_{\bb{\vartheta}^0}} ( t_1)
|^2
W(t_1,\boldsymbol t_2) dt_1d\boldsymbol t_2>0.
\end{equation}
  Let  assumptions (A.1)--(A.4),  (A.8) and (A.9) be satisfied and  let also (A.5), (A.6) with $s_1=s_2=1  $  and $\bb{\vartheta}_0$ replaced by $\bb{\vartheta}^0$. As soon as $n\to \infty$,
  $$
  \Delta_{n,W} \stackrel{P}{\to} \infty.
  $$
  \end{thm}

  The proof is deferred to the Appendix.

  Theorems 1 and 2 imply consistency of the test and also that large values of $\Delta_{n,W}$ indicate that the null hypothesis is violated.

\section{Computations and resampling}\label{sec_3}


\subsection{Computations}\label{sec_31}
Following Hu\v{s}kov\'a et al. (2018), we impose the decomposition $W(t_1,\boldsymbol t_2)=w_1(t_1)w_2(\boldsymbol t_2)$ on the weight function. If in addition the individual weight functions $w_m(\cdot), \ m=1,2$ satisfy all other requirements stated in Assumption (A.8) of the Appendix, then the test statistic in \eqref{testst} takes the form
\begin{equation} \label{testsum}
\Delta_{n,W}=\frac{1}{n} \sum_{j,k=1}^n I_{1,jk}I_{2,jk}+\frac{1}{n^3} \sum_{j,k=1}^n I_{1,jk}  \sum_{j,k=1}^n I_{2,jk}- \frac{2}{n^2} \sum_{j,k,\ell=1}^n I_{1,jk}I_{2,j\ell},
\end{equation}
where $I_{1,jk}:=I_{w_1}(\widehat \varepsilon_{jk})$, $I_{2,jk}:=I_{w_2}(\boldsymbol X_{jk})$, with $\boldsymbol X_{jk}=\boldsymbol X_{j}-\boldsymbol X_{k}$  and
$\widehat \varepsilon_{jk}=\widehat \varepsilon_{j}-\widehat \varepsilon_{k}, \ j,k=1,\ldots,n$, and
\begin{equation} \label{INTK}
I_{w_m}(\boldsymbol x) = \int \cos(\boldsymbol t^\top \boldsymbol x)w_m(\boldsymbol t){\rm{d}}\boldsymbol t, \quad m=1,2.
\end{equation}

The weight function $w_m(\cdot)$ in \eqref{INTK} may be chosen in a way that facilitates integration which is extremely important in high dimension. To this end notice that if $w_m(\cdot)$ is replaced by a spherical density, then the right-hand side of \eqref{INTK} gives (by definition) the CF corresponding to $w_m(\cdot)$ computed at the argument $\boldsymbol x$. Furthermore recall   that within the class of all spherical distributions, the integral in \eqref{INTK} depends on $\boldsymbol x$ only via its usual Euclidean norm $\|\boldsymbol x\|$, and specifically  $I_{w_m}(\boldsymbol x)=\Psi(\|\boldsymbol x\|)$, where the functional form of the univariate function $\Psi(\cdot)$ depends on the underlying subfamily of spherical distributions. In this connection $\Psi(\cdot)$ is called the ``characteristic kernel'' of the particular subfamily; see \citet{fang1990}. Consequently the test statistic in \eqref{testsum} becomes a function of $\Psi(\|\boldsymbol x\|)$ alone. Subfamilies of spherical distributions with simple kernels is the class of spherical stable distributions with $\Psi^{(S)}_\gamma(u)=e^{-u^\gamma}$, $0<\gamma \leq 2$,  and the class of generalized multivariate Laplace distributions  with $\Psi^{(L)}_\gamma(u)=(1+u^2)^{-\gamma}, \ \gamma>0$. For more information on these particular cases the reader is referred to \citet{nolan2013}, and to \citet{kozubowski2013}, respectively. For further use we simply note that interesting special cases of spherical stable distributions are the Cauchy distribution and the normal distribution corresponding to $\Psi^{(S)}_\gamma$ with $\gamma=1$ and $\gamma=2$, respectively, while the classical multivariate Laplace distribution results from  $\Psi^{(L)}_\gamma$ for $\gamma=1$.

\subsection{Resampling}\label{sec_32}
Recall that the null hypothesis ${\cal{H}}_0$ in ({\ref{null}) corresponds to model \eqref{modelI} in which both the true value of transformation parameter $\boldsymbol \vartheta$ as well as the error density are unknown. In this connection, and since, as was noted in Section 2, the asymptotic distribution of the test criterion under the null hypothesis depends on these quantities, among other things, we provide here a resampling scheme which can be used in order to compute critical points and actually carry out the test. The resampling scheme, which was proposed by \citet{neumeyer2016}, involves resampling from the observed ${\bm {X}}_j$ and independently constructing the bootstrap errors by smoothing the residuals. The bootstrap model then fulfils the null hypothesis since
\[
\frac{\mathcal{T}_{\widehat{\bb{\vartheta}}}(Y_j^* )-\Expec^*\mathopen{}\big( \mathcal{T}_{\widehat{\bb{\vartheta}}}(Y_j^* ) \,\big|\, {\boldsymbol X}_j^* \big)}
{\sqrt{\Var^*\mathopen{}\big( \mathcal{T}_{\widehat{\bb{\vartheta}}}(Y_j^* ) \,\big|\, {\boldsymbol X}_j^* \big)}}
:= \frac{\varepsilon^*_j}{\sqrt{1+a_n^2}} \; \bot^* \; {\boldsymbol X}_j^*,
\]
where $\mathbb \Expec^*$ and $\mathbb \Var^*$ denotes the conditional expectation
and variance and $\bot^*$ the conditional independence given the original sample.

We now describe the resampling procedure. Let $a_n$ be a positive smoothing parameter such that
$a_n \to 0$ and $n a_n \to \infty$, as $n\to\infty$.
Also, denote by $\{\xi_j\}_{j=1}^n$ a sequence of  random variables which are drawn independently of any other stochastic
quantity involved in the test criterion.
The bootstrap procedure is as follows:
\begin{enumerate}
\item Draw $\bb X_1^*,\dots,\bb X_n^*$ with replacement from $\bb X_1,\dots,\bb X_n$.
\item\label{itm:bootstraperrors} Generate i.i.d.\ random variables $\{\xi_j\}_{j=1}^n$ with a standard normal distribution
			and let $\varepsilon^*_j=a_n\xi_j+\widehat \varepsilon_j, \ j=1,...,n$, with
			$\widehat \varepsilon_j$ defined in \eqref{residual}.
\item Compute the bootstrap responses
			$Y^*_j={\cal {T}}^{-1}_{\widehat{{\bb{\vartheta}}}}(\widehat m_{\widehat{\bm{\vartheta}}}({\bb X}_{j}^*)+\widehat\sigma_{\widehat{\bm\vartheta}}({\bb X}_{j}^*)\varepsilon^*_j)$, $j=1,\ldots,n$.
\item On the basis of the observations $(Y^*_j,\bb X_j^*)$, $j=1,\dots,n$, refit the model and obtain the bootstrap residuals $\widehat \varepsilon^*_j, \ j=1,\ldots,n$.
\item Calculate the value of the test statistic, say $\Delta_{n,W}^*$, corresponding to the bootstrap sample $(Y^*_j,\bb X_j^*)$, $j=1,\dots,n$.
\item Repeat the previous steps a number of times, say $B$, and obtain $\{\Delta_{n,W}^{*(b)}\}_{b=1}^B$.
\item Calculate the critical point of a size-$\alpha$ test as the $(1-\alpha)$ level quantile $c^*_{1-\alpha}$ of $\Delta_{n,W}^{*(b)}$, $b=1,...,B$.
\item Reject the null hypothesis if $\Delta_{n,W}>c^*_{1-\alpha}$, where $\Delta_{n,W}$ is the value of the test statistic based on the original observations $(Y_j,\bb{X}_j)$, $j=1,\dots,n$.
\end{enumerate}

\section{Simulations}
In this section we present the results of a Monte Carlo exercise that sheds light on the small-sample properties of the new test statistic
and compare our test with the classical Kolmogorov--Smirnov (later denoted by KS) and Cram\'er--von Mises (CM) criteria suggested by \citet{neumeyer2016}.
We considered the family of transformations
$$
{\cal{T}}_\vartheta (Y) =
\begin{cases}
~\left\{ (Y + 1)^\vartheta - 1\right\}/ \vartheta 					& \text{if } Y\geq 0,\, \vartheta \ne 0, \\
~\log(Y + 1) 																								 & \text{if } Y\geq 0,\, \vartheta = 0, \\
~-\left\{ (-Y + 1)^{2-\vartheta} - 1\right\}/ (2-\vartheta) & \text{if } Y < 0,\, \vartheta \ne 2, \\
~-\log(-Y + 1)																							 & \text{if } Y < 0,\, \vartheta = 2,
\end{cases}
$$
proposed by \citet{yeo2000}, and randomly generated paired observations $(Y_j,X_j)$, $j=1,\ldots,n$,
from the univariate heteroskedastic model
\begin{equation}\label{MCmodel}
\mathcal{T}_{\vartheta_0} (Y) = m(X) + \sigma(X) \varepsilon,
\end{equation}
where $\vartheta_0=0$, $m(x)=1.5+\exp(x)$ and $\sigma(x)=x$.
Here, $\varepsilon$ is an error term which should be stochastically independent of $X$ under the null hypothesis.
The distribution of the covariate $X$
and the distribution of the error $\varepsilon$ (conditional on $X$) were chosen as one of the following:
\begin{enumerate}[label={\bfseries Model \Alph*.},labelsep=8pt,
									labelindent=.5\parindent,
									itemindent=2\parindent,
									leftmargin=*,
									before=\setlength{\listparindent}{-\leftmargin}]
	\item $X\sim\operatorname{uniform}(0,1)$. Let $\operatorname{ST}(\zeta,\omega,\eta,\nu)$ denote the univariate skew-$t$ distribution with parameters
				$\zeta$ (location), $\omega$ (scale), $\eta$ (shape) and $\nu$ (degrees of freedom) as defined by \citet{azzalini2005}.
				Define
				\[
					(\varepsilon \,\big|\, X=x)
					\stackrel{d}{=}
					\begin{cases}
						 \frac{W_{\eta,\nu}-\Expec(W_{\eta,\nu})}{\sqrt{\Var(W_{\eta,\nu})}} & \text{if } 0 \le x \le 0.5, \\
						Z & \text{if } 0.5 < x \le 1,
					\end{cases}
				\]
				where $W_{\eta,\nu}\sim\operatorname{ST}(0,1,\eta,\nu)$ and $Z\sim\operatorname{N}(0,1)$,
				both quantities independent of $X$.
				Notice that the null hypothesis of a heteroskedastic transformation structure
				is violated except when $\eta \to 0$ and $\nu \to \infty$,
				in which case $W_{\eta,\nu} \to \operatorname{N}(0,1)$ so that
				$\varepsilon$ and $X$ are independent.
	\item $X\sim\operatorname{uniform}(0,1)$. Define
				\[
					(\varepsilon \,\big|\, X=x)
					\stackrel{d}{=}
					\begin{cases}
						\frac{W_\nu-\nu}{\sqrt{2\nu}} & \text{if } 0 \le x \le 0.5, \\
						Z & \text{if } 0.5 < x \le 1,
					\end{cases}
				\]
				where $W_\nu\sim\chi_\nu^2$ and $Z\sim\operatorname{N}(0,1)$.
				Notice that $\varepsilon$ is stochastically dependent on $X$
				except when $\nu\to\infty$, in which case the null hypothesis of a
				heteroskedastic transformation structure is satisfied.
	\item $X\sim\operatorname{uniform}(0,1)$. Let $\operatorname{AL}(\nu,\lambda,\kappa)$ denote the univariate asymmetric Laplace distribution with parameters $\nu$ (location),
				$\lambda$ (scale) and $\kappa$ (shape) as studied by \citet{kozubowski2013}.
				Define
				\[
				(\varepsilon \,\big|\, X=x)
				\stackrel{d}{=}
				\begin{cases}
				\frac{W_\kappa-(1-\kappa^2)/\kappa}{\sqrt{(1+\kappa^4)/\kappa^2}} & \text{if } 0 \le x \le 0.5, \\
				Z & \text{if } 0.5 < x \le 1,
				\end{cases}
				\]
				where $W_\kappa\sim\operatorname{AL}(0,1,\kappa)$ and $Z\sim\operatorname{AL}(0,1,1)$,
				the latter being an observation from the usual symmetric Laplace distribution.
				Notice that $\varepsilon$ is stochastically dependent on $X$
				except when $\kappa=1$, in which case the null hypothesis is satisfied.
	\item To investigate the behavior of the tests also in the case of a discrete covariate, we
				considered generating $X$ from a discrete uniform distribution on the set
				$\{\tfrac{1}{10},\tfrac{2}{10},\cdots,1\}$, with $\varepsilon$ having the
				same distribution as the errors of Model A given above.
\end{enumerate}

For a test size of $\alpha=0.05$, the rejection frequency of the test was recorded
for sample sizes $n=100,200,300$.
The bootstrap resampling scheme of Section~\ref{sec_32} requires a choice of the smoothing parameter $a_n$,
where we followed \citet{neumeyer2016} and chose $a_n=0.5n^{-1/4}$ throughout.
Since the bootstrap replications are time consuming we have employed the warp-speed method of \citet{giacomini2013} in order to calculate critical points of the test criterion. With this method we generate only one bootstrap resample for each Monte Carlo sample and thereby compute the bootstrap test statistic $\Delta_{n,W}^*$ for that resample. Then, for a number $M$ of Monte Carlo replications, the  size-$\alpha$ critical point is determined similarly as in step 7 of Section \ref{sec_32}, by computing the $(1-\alpha)$-level quantile of $\Delta_{n,W}^{*(m)}, \ m=1,...,M$. For all simulations the number of Monte Carlo replications was set to $M=5\,000$.

\subsubsection*{Estimation of the tranformation parameter}
To estimate the transformation parameter $\vartheta_0$ in \eqref{MCmodel} we employ
the profile likelihood estimator recently studied by \citet{neumeyer2016}, which allows for the
heteroskedastic error structure present in our setup.
Implementation of this estimator relies on some practical considerations, which we now discuss.
The estimator involves estimating $m(\cdot)$ and $\sigma(\cdot)$ nonparametrically,
for which we used local linear regression with a Gaussian kernel and bandwidth
chosen by the direct plug-in methodology proposed by \citet{ruppert1995}.
The estimator also requires estimation of the density of the regression errors.
For this purpose we used a Gaussian kernel with bandwidth chosen by the method of \citet{sheather1991}.
Both these methods of bandwidth selection have been implemented by \citet{KernSmooth}
in the \textsf{R} package \texttt{KernSmooth}.
The author warns that in some cases these procedures might be numerically
unstable (see p.~10 of \citealp{KernSmooth}). In these rare situations we turned
to simple rule of thumb selection methods: in the case of nonparametric regression
we used the rule of thumb of \citet{fan1996} (implemented in the package \texttt{locpol} by \citealp{locpol})
and in the case of density estimation a rule of \citet[equations (3.28) and (3.30) on pp.~45 and 47]{silverman1986}.
Finally, to actually implement nonparametric regression and density estimation using these
chosen bandwidths, we employed the package \texttt{np} \citep{np} designed specifically for this purpose.

\subsection{Simulation results for independence}
The simulation results for the three considered models are shown in Tables~\ref{tabA} to \ref{tabD}.
The results for the classical Kolmogorov--Smirnov and Cram\'er--von Mises tests
are given in the columns labelled KS and CM, respectively.
The percentage of rejections of our statistic $\Delta_{n,W}$ is given for
three different choices of the characteristic kernel $\Psi(\cdot)$ discussed Section~\ref{sec_31},
with various choices of a tuning parameter $c>0$. Specifically and for $c>0$, we use as weight functions scaled spherical stable densities that yield $\Psi^{(S)}_\gamma=e^{-c u^\gamma}$ (recall that $\gamma=2$ coincides with the Gaussian case), and scaled spherical Laplace densities that yield  $\Psi^{(S)}_\gamma=(1+(u^2/c))^{-\gamma}$.  
From Tables~\ref{tabA} to \ref{tabD} it is clear that the test based on $\Delta_{n,W}$ respects
the nominal size well. However, for smaller sample sizes the test appears to be slightly
conservative in some cases. The same can be said of the classical KS and CM tests.

Notice that, under the various considered alternatives, the power of all tests increase in accordance
with the nature of the dependence that is introduced between the covariates and the error terms.
Moreover, in agreement with the consistency of the test established formally in Theorem~\ref{thm2},
under alternatives the power of our test appears to increase as the sample size increases.
Overall, in terms of power the new test based on $\Delta_{n,W}$ exhibits competitive
performance and even outperforms the classical tests for most considered choices of the tuning parameter~$c$.

We close by noting that the value of the tuning parameter $c$ clearly has some effect on
the power of the test based on $\Delta_{n,W}$. There exist several interpretations regarding the value of $c$
and for more information on this the reader is referred to the recent review paper by \citet{meintanis2016}.
For all tests we chose values of $c$ for which the tests exhibit good size properties
as well as good power under alternatives.

\begin{table}[tp]
	\centering\scriptsize
	\setlength\tabcolsep{4pt}
	\caption{Size and power results for verifying the validity of Model A. The null hypothesis is satisfied for $\eta=0$ and $\nu=\infty$. The nominal size of the test is $\alpha=0.05$.}\label{tabA}
	 \begin{tabular}{ccccccccccccccccccccccccccc}
		\toprule
		& & & & & \multicolumn{12}{c}{$\Delta_{n,W}$} \\ \cmidrule(lr){6-17}
		& & & & & \multicolumn{4}{c}{$\Psi_2^{(S)}(u)=\exp(-cu^2)$} & \multicolumn{4}{c}{$\Psi_1^{(S)}(u)=\exp(-cu)$} & \multicolumn{4}{c}{$\Psi_1^{(L)}(u)=(1+u^2/c)^{-1}$} \\ \cmidrule(lr){6-9}\cmidrule(lr){10-13}\cmidrule(lr){14-17}
		$\eta$ & $\nu$ & $n$ & KS & CM & {$c=1$} & {$c=1.5$} & {$c=2$} & {$c=4$} & {$c=0.5$} & {$c=1$} & {$c=1.5$} & {$c=2$} & {$c=0.1$} & {$c=0.25$} & {$c=0.5$} & {$c=1$} \\
\midrule
0     & $\infty$ & 100   & 4.7   & 3.9   & 4.5   & 4.4   & 4.4   & 4.2   & 3.8   & 4.0   & 3.7   & 3.6   & 3.9   & 4.2   & 4.3   & 4.3 \\
      &       & 200   & 4.6   & 4.2   & 4.6   & 4.9   & 4.8   & 4.0   & 3.9   & 3.7   & 3.8   & 3.7   & 3.5   & 4.4   & 4.6   & 4.9 \\
      &       & 300   & 5.0   & 4.9   & 5.8   & 5.9   & 6.1   & 5.6   & 4.9   & 5.0   & 4.9   & 4.8   & 5.0   & 5.6   & 5.7   & 5.6 \\\cmidrule(lr){1-17}
0     & 5     & 100   & 7.1   & 9.3   & 10.4  & 9.9   & 9.4   & 8.7   & 8.9   & 8.7   & 8.5   & 8.0   & 7.6   & 8.6   & 9.2   & 10.2 \\
      &       & 200   & 11.4  & 15.5  & 18.3  & 18.7  & 18.6  & 17.3  & 16.5  & 17.7  & 17.0  & 15.9  & 14.4  & 17.5  & 18.1  & 18.4 \\
      &       & 300   & 15.7  & 21.5  & 26.5  & 27.9  & 27.5  & 25.0  & 24.6  & 24.5  & 24.3  & 23.9  & 22.5  & 25.6  & 27.2  & 26.3 \\\cmidrule(lr){1-17}
0     & 2.1   & 100   & 22.3  & 24.5  & 31.2  & 31.4  & 31.2  & 31.2  & 29.4  & 30.3  & 30.5  & 30.2  & 29.4  & 31.5  & 31.8  & 31.6 \\
      &       & 200   & 41.0  & 44.2  & 51.5  & 52.9  & 53.4  & 54.2  & 52.0  & 53.4  & 54.7  & 54.1  & 53.6  & 54.2  & 53.6  & 52.3 \\
      &       & 300   & 56.4  & 60.5  & 69.3  & 70.3  & 71.7  & 71.4  & 69.2  & 70.7  & 71.0  & 71.4  & 70.4  & 71.5  & 71.2  & 70.0 \\\cmidrule(lr){1-17}
100   & 2.1   & 100   & 33.0  & 35.8  & 39.8  & 40.2  & 40.4  & 41.8  & 42.0  & 43.2  & 44.3  & 44.2  & 44.2  & 42.5  & 41.1  & 40.4 \\
      &       & 200   & 51.3  & 52.2  & 57.5  & 58.1  & 58.2  & 58.9  & 59.6  & 60.9  & 61.3  & 61.7  & 61.0  & 59.9  & 58.9  & 57.5 \\
      &       & 300   & 67.1  & 66.8  & 67.1  & 68.5  & 69.6  & 71.8  & 71.1  & 72.0  & 73.3  & 74.7  & 75.0  & 71.9  & 70.2  & 68.2 \\
		\bottomrule
	\end{tabular}
\end{table}

\begin{table}[tp]
	\centering\scriptsize
	\setlength\tabcolsep{4pt}
	\caption{Size and power results for verifying the validity of Model B. The null hypothesis is satisfied for $\nu=\infty$. The nominal size of the test is $\alpha=0.05$.}\label{tabB}
	 \begin{tabular}{ccccccccccccccccccccccccccccccccc}
		\toprule
		& & & & \multicolumn{12}{c}{$\Delta_{n,W}$} \\ \cmidrule(lr){5-16}
		& & & & \multicolumn{4}{c}{$\Psi_2^{(S)}(u)=\exp(-cu^2)$} & \multicolumn{4}{c}{$\Psi_1^{(S)}(u)=\exp(-cu)$} & \multicolumn{4}{c}{$\Psi_1^{(L)}(u)=(1+u^2/c)^{-1}$} \\ \cmidrule(lr){5-8}\cmidrule(lr){9-12}\cmidrule(lr){13-16}
		$\nu$ & $n$ & KS & CM & {$c=1$} & {$c=1.5$} & {$c=2$} & {$c=4$} & {$c=0.5$} & {$c=1$} & {$c=1.5$} & {$c=2$} & {$c=0.1$} & {$c=0.25$} & {$c=0.5$} & {$c=1$} \\
\midrule
$\infty$ & 100   & 5.1   & 3.8   & 5.1   & 4.9   & 4.6   & 4.4   & 4.4   & 4.5   & 4.4   & 4.1   & 4.7   & 5.0   & 5.0   & 4.9 \\
      & 200   & 3.9   & 4.1   & 4.3   & 4.2   & 4.1   & 4.0   & 3.7   & 3.6   & 3.5   & 3.4   & 4.0   & 4.4   & 4.3   & 4.2 \\
      & 300   & 4.7   & 4.7   & 5.9   & 5.5   & 5.3   & 5.1   & 4.8   & 4.8   & 4.5   & 4.6   & 5.3   & 5.5   & 5.7   & 5.8 \\\cmidrule(lr){1-16}
10    & 100   & 5.9   & 7.1   & 8.9   & 8.4   & 7.9   & 6.8   & 7.7   & 7.2   & 6.9   & 6.6   & 8.1   & 8.5   & 9.1   & 9.5 \\
      & 200   & 9.7   & 10.2  & 13.5  & 12.7  & 12.2  & 11.0  & 11.6  & 11.4  & 10.9  & 10.3  & 12.2  & 13.0  & 13.1  & 13.6 \\
      & 300   & 10.0  & 11.0  & 14.4  & 13.4  & 12.6  & 11.4  & 12.9  & 12.9  & 12.4  & 12.0  & 13.2  & 14.0  & 14.3  & 14.7 \\\cmidrule(lr){1-16}
5     & 100   & 9.7   & 10.4  & 13.3  & 12.0  & 11.4  & 11.2  & 12.2  & 12.2  & 11.9  & 11.4  & 11.8  & 12.8  & 13.3  & 14.0 \\
      & 200   & 16.1  & 17.6  & 21.6  & 20.8  & 20.5  & 19.7  & 21.6  & 22.1  & 20.9  & 20.6  & 21.4  & 22.0  & 22.5  & 21.9 \\
      & 300   & 14.7  & 18.1  & 23.7  & 23.0  & 22.5  & 22.1  & 23.6  & 23.4  & 23.0  & 22.9  & 23.1  & 23.6  & 23.9  & 24.1 \\\cmidrule(lr){1-16}
3     & 100   & 14.4  & 17.4  & 22.7  & 21.5  & 21.1  & 18.7  & 21.1  & 21.1  & 20.2  & 19.0  & 21.5  & 22.2  & 22.3  & 22.5 \\
      & 200   & 22.5  & 26.1  & 30.6  & 30.6  & 30.4  & 30.6  & 32.3  & 32.6  & 32.5  & 32.5  & 31.5  & 31.7  & 31.7  & 31.0 \\
      & 300   & 31.6  & 35.4  & 36.6  & 36.6  & 37.3  & 39.2  & 38.8  & 39.5  & 39.6  & 40.4  & 38.0  & 37.1  & 36.7  & 36.5 \\\cmidrule(lr){1-16}
2     & 100   & 21.0  & 25.2  & 30.2  & 29.6  & 28.7  & 28.3  & 29.6  & 30.3  & 30.0  & 29.7  & 29.8  & 29.9  & 29.9  & 29.2 \\
      & 200   & 39.1  & 43.0  & 47.0  & 47.9  & 48.1  & 50.7  & 49.9  & 51.5  & 52.4  & 52.7  & 49.5  & 47.9  & 46.9  & 46.2 \\
      & 300   & 53.0  & 52.2  & 53.8  & 54.9  & 56.3  & 59.4  & 57.6  & 59.3  & 60.4  & 61.3  & 57.4  & 54.6  & 53.7  & 52.6 \\
		\bottomrule
	\end{tabular}
\end{table}

\begin{table}[tp]
	\centering\scriptsize
	\setlength\tabcolsep{4pt}
	\caption{Size and power results for verifying the validity of Model C. The null hypothesis is satisfied for $\kappa=1$. The nominal size of the test is $\alpha=0.05$.}\label{tabC}
	 \begin{tabular}{ccccccccccccccccccccccccccccccccc}
		\toprule
		& & & & \multicolumn{12}{c}{$\Delta_{n,W}$} \\ \cmidrule(lr){5-16}
		& & & & \multicolumn{4}{c}{$\Psi_2^{(S)}(u)=\exp(-cu^2)$} & \multicolumn{4}{c}{$\Psi_1^{(S)}(u)=\exp(-cu)$} & \multicolumn{4}{c}{$\Psi_1^{(L)}(u)=(1+u^2/c)^{-1}$} \\ \cmidrule(lr){5-8}\cmidrule(lr){9-12}\cmidrule(lr){13-16}
		$\kappa$ & $n$ & KS & CM & {$c=1$} & {$c=1.5$} & {$c=2$} & {$c=4$} & {$c=1$} & {$c=1.5$} & {$c=2$} & {$c=4$} & {$c=0.01$} & {$c=0.025$} & {$c=0.05$} & {$c=0.1$} \\
\midrule
1 & 100 & 5.1 & 4.8 & 5.0 & 5.5 & 5.4 & 5.2 & 5.0 & 4.8 & 4.7 & 4.7 & 4.0 & 4.8 & 5.0 & 4.6 \\
& 200 & 5.9 & 6.2 & 5.2 & 5.3 & 5.5 & 5.0 & 5.2 & 4.9 & 4.9 & 4.6 & 4.3 & 4.7 & 4.7 & 4.9 \\
& 300 & 6.5 & 5.5 & 5.8 & 5.8 & 5.4 & 5.5 & 4.9 & 5.2 & 5.2 & 5.3 & 5.4 & 5.6 & 5.0 & 4.9 \\\cmidrule(lr){1-16}
2 & 100 & 8.6 & 9.4 & 13.9 & 14.1 & 13.9 & 13.9 & 13.6 & 14.0 & 13.5 & 11.5 & 10.0 & 11.1 & 11.9 & 13.3 \\
& 200 & 20.7 & 23.8 & 30.2 & 30.5 & 30.7 & 31.9 & 29.9 & 31.4 & 31.1 & 30.8 & 24.9 & 30.1 & 30.8 & 32.0 \\
& 300 & 34.0 & 32.3 & 33.3 & 35.1 & 35.7 & 41.0 & 37.2 & 39.3 & 38.6 & 44.3 & 39.8 & 43.6 & 44.1 & 40.3 \\\cmidrule(lr){1-16}
5 & 100 & 10.7 & 12.1 & 18.3 & 18.6 & 19.2 & 19.8 & 18.6 & 19.2 & 20.2 & 19.6 & 16.1 & 17.6 & 19.4 & 20.5 \\
& 200 & 27.8 & 33.4 & 38.7 & 40.3 & 41.7 & 44.2 & 42.3 & 43.3 & 44.7 & 45.6 & 41.7 & 44.5 & 46.5 & 45.9 \\
& 300 & 41.3 & 41.1 & 43.6 & 45.7 & 47.0 & 57.0 & 53.5 & 56.0 & 59.2 & 66.7 & 63.6 & 65.8 & 66.9 & 64.1 \\
		\bottomrule
	\end{tabular}
\end{table}

\begin{table}[tp]
	\centering\scriptsize
	\setlength\tabcolsep{4pt}
	\caption{Size and power results for verifying the validity of Model D. The null hypothesis is satisfied for $\eta=0$ and $\nu=\infty$. The nominal size of the test is $\alpha=0.05$.}\label{tabD}
	 \begin{tabular}{ccccccccccccccccccccccccccccccccc}
		\toprule
		& & & & & \multicolumn{12}{c}{$\Delta_{n,W}$} \\ \cmidrule(lr){6-17}
		& & & & & \multicolumn{4}{c}{$\Psi_2^{(S)}(u)=\exp(-cu^2)$} & \multicolumn{4}{c}{$\Psi_1^{(S)}(u)=\exp(-cu)$} & \multicolumn{4}{c}{$\Psi_1^{(L)}(u)=(1+u^2/c)^{-1}$} \\ \cmidrule(lr){6-9}\cmidrule(lr){10-13}\cmidrule(lr){14-17}
		$\eta$ & $\nu$ & $n$ & KS & CM & {$c=1$} & {$c=1.5$} & {$c=2$} & {$c=4$} & {$c=0.5$} & {$c=1$} & {$c=1.5$} & {$c=2$} & {$c=0.1$} & {$c=.25$} & {$c=.5$} & {$c=1$} \\
\midrule
0     & $\infty$ & 100   & 3.1   & 2.8   & 3.9   & 3.5   & 3.3   & 3.8   & 3.6   & 3.5   & 3.5   & 3.4   & 3.7   & 3.7   & 3.5   & 3.4 \\
      &       & 200   & 3.7   & 3.5   & 4.4   & 4.6   & 4.3   & 4.0   & 4.2   & 4.5   & 4.7   & 4.2   & 3.9   & 4.1   & 4.4   & 4.5 \\
      &       & 300   & 4.5   & 4.0   & 4.6   & 4.8   & 4.6   & 4.8   & 4.8   & 4.5   & 4.5   & 4.6   & 4.8   & 4.9   & 4.5   & 4.4 \\\cmidrule(lr){1-17}
0     & 5     & 100   & 5.7   & 5.8   & 7.4   & 7.6   & 7.1   & 6.8   & 7.1   & 7.1   & 7.3   & 6.8   & 6.3   & 6.9   & 7.0   & 7.3 \\
      &       & 200   & 7.2   & 8.6   & 9.8   & 10.0  & 10.4  & 10.4  & 10.8  & 11.1  & 11.0  & 10.6  & 10.6  & 10.7  & 11.0  & 10.5 \\
      &       & 300   & 9.7   & 11.2  & 11.1  & 12.1  & 13.2  & 13.3  & 13.1  & 13.7  & 14.0  & 13.3  & 12.9  & 13.6  & 13.7  & 13.8 \\\cmidrule(lr){1-17}
0     & 2.1   & 100   & 18.9  & 19.0  & 22.1  & 23.0  & 25.3  & 26.2  & 25.5  & 25.8  & 26.1  & 27.2  & 25.7  & 26.6  & 26.6  & 25.7 \\
      &       & 200   & 35.3  & 33.3  & 39.4  & 42.7  & 46.7  & 48.6  & 46.6  & 48.8  & 49.9  & 50.9  & 48.3  & 49.7  & 49.8  & 47.8 \\
      &       & 300   & 53.4  & 49.1  & 56.1  & 60.4  & 64.9  & 67.5  & 64.4  & 66.6  & 68.1  & 67.9  & 68.0  & 68.0  & 67.4  & 65.7 \\\cmidrule(lr){1-17}
100   & 2.1   & 100   & 30.1  & 27.7  & 32.1  & 34.5  & 36.0  & 37.7  & 38.2  & 38.2  & 38.3  & 40.1  & 38.8  & 39.4  & 38.4  & 36.6 \\
      &       & 200   & 53.9  & 53.1  & 54.2  & 57.5  & 61.5  & 63.9  & 60.3  & 62.3  & 63.7  & 65.8  & 66.8  & 65.6  & 63.5  & 61.6 \\
      &       & 300   & 69.1  & 65.2  & 64.8  & 68.2  & 72.6  & 75.8  & 70.8  & 72.1  & 73.5  & 76.8  & 77.6  & 76.5  & 74.3  & 72.2 \\
		\bottomrule
	\end{tabular}
\end{table}

\subsection{Simulation results for normality and symmetry}
One of the main goals of transformation is to reduce skewness and possibly even achieve near normality. These issues have been recently investigated by \citet{yeo2000}, \citet{yeo2014}, \citet{meintanis2015}, and \citet{chen2002}, with or without regressors, with the last reference also providing asymptotics for a test of normality of the after Box--Cox transformation errors under homoskedasticity. In this section we investigate how the CF tests for symmetry and normality designed for i.i.d.\ data perform within the significantly more complicated context of the semiparametric heteroskedastic transformation model \eqref{modelI}. In this connection we note that such CF tests have already shown competitive performance in more classical regression frameworks; see \citet{huskova2010,huskova2012}. The CF test statistics of normality and symmetry are motivated by uniqueness of the CF of any given distribution, and by the fact that for any zero-symmetric distribution, the imaginary part of its CF is identically equal to zero. Thus we have the test statistic for normality 
\begin{equation}\label{test.gauss}
\Delta^{(G)}_{n,w}= n \int_{-\infty}^\infty |\widehat \varphi_{\widehat \varepsilon}(t)-e^{-t^2/2}|^2 w(t){\rm{d}}t,
\end{equation}
and the test statistic for symmetry of errors
\begin{equation}\label{test.symm}
\Delta^{(S)}_{n,w}= n \int_{-\infty}^\infty \left({\texttt{Im}} (\widehat \varphi_{\widehat \varepsilon}(t)\right)^2 w(t){\rm{d}}t,
\end{equation}
where $w(\cdot)$ is a weight function analogous to $w_1(\cdot)$ of assumption (A.8) and ${\texttt{Im}}(z)$ denotes the imaginary part of a complex number $z$.

In our Monte Carlo simulations, the results of which are shown in Tables~\ref{tabA2} to \ref{tab.symm.D},
we use three choices of the weight function $w(\cdot)$ and various choices of the tuning parameter $c$ for both tests.
To obtain the critical value of the \emph{normality} test corresponding to the statistic in \eqref{test.gauss}, we used the
same bootstrap resampling scheme as given in Section~\ref{sec_32}, but with step \ref{itm:bootstraperrors} replaced by:
\begin{enumerate}[start=2]
	\item[$2.\hspace{-2pt}'$] Generate i.i.d.\ errors $\{\varepsilon^*_j\}_{j=1}^n$ from a standard normal distribution.
\end{enumerate}
The critical value of test for \emph{symmetry} based on the statistic in \eqref{test.symm} was obtained using a wild bootstrap scheme, see \citet{NZ2000}, \citet{delgado2001} and \cite{huskova2012}, 
which is the same as that given in Section~\ref{sec_32} but with step \ref{itm:bootstraperrors} replaced by:
\begin{enumerate}[start=2]
	\item[$2.\hspace{-2pt}''$] Generate i.i.d.\ random variables $\{U_j\}_{j=1}^n$ according to the law $\Prob(U_j=+1)=\Prob(U_j=-1)=\tfrac{1}{2}$
				and set $\varepsilon^*_j=U_j\tilde \varepsilon_j$, $j=1,...,n$, where the $\tilde\varepsilon_j$ are drawn randomly with replacement
				from $\widehat\varepsilon_1,\ldots,\widehat\varepsilon_n$.
\end{enumerate}

Firstly, concerning the test for normality, we see from Tables~\ref{tabA2} to \ref{tabD2} 
that the size of the test is approximately around the nominal size, being slightly
conservative in some cases. The power of the test increases with the extend of violation of normality,
i.e., as the skewness parameter $\eta$ is increased or as the degrees of freedom parameter $\nu$ is decreased.
Finally, the results seem to suggest consistency of the test in the sense that, for each given fixed alternative, the
power increases gradually as the sample size is increased. These observations hold for error terms
generated under Models A, B and D. For Model C, we just note that an analogous CF-based test
can be constructed along the lines of \citet{meintanis2004}.

Regarding the test for symmetry similar conclusions as above can be made
(see Tables~\ref{tab.symm.A} to \ref{tab.symm.D}).
Note however that for Model A there is a clear over-rejection of the null hypothesis
of symmetry in the cases where $\eta=0$ and $\nu=2.1$ or $5$, i.e.\ when
the error distribution is symmetric but heavy-tailed.
As seen in Table~\ref{tab.symm.C} this is also true for Model C in the case where $\kappa=1$, although to a lesser extent.
To address this issue we employed the permutation test suggested by
\citet{henze2003} developed specifically to address the issue of over-rejection.
The results obtained in this way, however, agree almost exactly with the
results in Tables~\ref{tab.symm.A} obtained using the wild bootstrap approach.
It should be noted that this issue of over-rejection does not occur when
the true transformation parameter $\vartheta_0$ is assumed to be known and
only arises in the more complicated setting where $\vartheta_0$ needs to be estimated.

In conclusion we note that the test for normality and the test for
symmetry both exhibit favourable properties even in this more complicated setting
of the heteroskedastic transformation model. However, our results are just indicative
of the performance of existing tests in this setting, and a more in depth study
is needed to explore the theoretical properties of these tests, which
might shed more light on some of the prevailing issues mentioned above.

\begin{table}[tp]
	\centering\scriptsize
	\setlength\tabcolsep{4pt}
	\caption{Size and power results for assessing normality ($\Delta^{(G)}_{n,w}$) of the error terms appearing in Model A. The nominal size of the test is $\alpha=0.05$.}\label{tabA2}
	 \begin{tabular}{cccccccccccccccccccccccccccccccc}
		\toprule
		& & & \multicolumn{4}{c}{$w(t)=\exp({-ct^2})$} & \multicolumn{4}{c}{$w(t)=({1+t^2/c^2})^{-1}$} & \multicolumn{4}{c}{$w(t)=\exp({-c|t|})$} \\ \cmidrule(lr){4-7}\cmidrule(lr){8-11}\cmidrule(lr){12-15}
		$\eta$ & $\nu$ & $n$ & {$c=0.1$} & {$c=0.25$} & {$c=0.5$} & {$c=1$} & {$c=0.5$} & {$c=1$} & {$c=1.5$} & {$c=2$} & {$c=0.05$} & {$c=0.1$} & {$c=0.25$} & {$c=0.5$} \\
\midrule
		0     & $\infty$ & 100   & 3.7   & 4.0   & 4.2   & 4.4   & 2.5   & 2.6   & 2.6   & 2.5   & 2.7   & 2.5   & 2.4   & 2.4 \\
		&       & 200   & 3.7   & 3.8   & 3.8   & 4.0   & 3.1   & 3.0   & 2.9   & 2.9   & 2.9   & 3.0   & 2.9   & 2.9 \\
		&       & 300   & 3.9   & 4.0   & 4.1   & 4.6   & 2.8   & 2.9   & 3.0   & 2.9   & 3.3   & 3.4   & 2.9   & 2.9 \\\cmidrule(lr){1-15}
		5     & $\infty$ & 100   & 15.2  & 15.5  & 15.2  & 13.9  & 13.5  & 14.5  & 14.6  & 15.2  & 14.7  & 15.8  & 16.0  & 14.7 \\
		&       & 200   & 30.4  & 30.8  & 30.9  & 31.3  & 22.7  & 24.8  & 25.9  & 26.3  & 27.2  & 27.4  & 27.2  & 26.4 \\
		&       & 300   & 41.8  & 41.5  & 41.9  & 40.7  & 30.7  & 33.5  & 35.0  & 36.3  & 39.4  & 39.7  & 38.2  & 36.6 \\\cmidrule(lr){1-15}
		20    & $\infty$ & 100   & 18.5  & 19.8  & 19.5  & 18.8  & 14.9  & 15.9  & 17.1  & 17.4  & 17.5  & 18.3  & 17.8  & 17.3 \\
		&       & 200   & 43.0  & 42.2  & 41.4  & 41.4  & 27.3  & 30.0  & 31.4  & 32.4  & 36.3  & 36.5  & 34.8  & 32.0 \\
		&       & 300   & 55.1  & 54.1  & 52.1  & 50.2  & 37.4  & 41.3  & 43.2  & 44.4  & 49.7  & 48.6  & 46.9  & 44.3 \\\cmidrule(lr){1-15}
		0     & 5     & 100   & 5.5   & 6.0   & 6.4   & 6.8   & 7.2   & 7.3   & 6.9   & 6.9   & 6.9   & 7.0   & 6.9   & 7.2 \\
		&       & 200   & 6.3   & 6.7   & 7.1   & 7.8   & 9.4   & 10.0  & 10.2  & 10.2  & 10.7  & 10.1  & 10.2  & 10.0 \\
		&       & 300   & 6.1   & 6.5   & 6.8   & 7.5   & 13.7  & 14.7  & 14.7  & 14.9  & 16.2  & 15.7  & 15.8  & 14.9 \\\cmidrule(lr){1-15}
		0     & 2.1   & 100   & 14.0  & 14.7  & 14.6  & 14.1  & 17.9  & 19.0  & 19.8  & 20.4  & 20.7  & 21.3  & 20.7  & 20.0 \\
		&       & 200   & 27.3  & 25.4  & 24.3  & 23.4  & 30.2  & 31.8  & 33.5  & 34.3  & 37.5  & 36.8  & 35.6  & 34.0 \\
		&       & 300   & 29.0  & 27.7  & 26.2  & 26.2  & 37.6  & 42.0  & 43.4  & 44.1  & 46.1  & 45.8  & 44.8  & 44.1 \\\cmidrule(lr){1-15}
		100   & 2.1   & 100   & 36.1  & 34.1  & 32.6  & 29.9  & 28.5  & 32.3  & 34.8  & 36.5  & 41.9  & 41.9  & 38.8  & 35.7 \\
		&       & 200   & 57.6  & 53.9  & 51.7  & 50.5  & 48.7  & 52.8  & 55.5  & 57.2  & 62.1  & 61.9  & 59.1  & 56.7 \\
		&       & 300   & 58.1  & 53.6  & 52.9  & 53.1  & 59.8  & 64.7  & 66.2  & 67.9  & 72.9  & 72.2  & 70.6  & 67.4 \\
		\bottomrule
	\end{tabular}
\end{table}

\begin{table}[tp]
	\centering\scriptsize
	\setlength\tabcolsep{4pt}
	\caption{Size and power results for assessing normality ($\Delta^{(G)}_{n,w}$) of the error terms appearing in Model B. The nominal size of the test is $\alpha=0.05$.}\label{tabB2}
	 \begin{tabular}{ccccccccccccccccccccccccccccc}
		\toprule
		& & \multicolumn{4}{c}{$w(t)=\exp({-ct^2})$} & \multicolumn{4}{c}{$w(t)=({1+t^2/c^2})^{-1}$} & \multicolumn{4}{c}{$w(t)=\exp({-c|t|})$} \\ \cmidrule(lr){3-6}\cmidrule(lr){7-10}\cmidrule(lr){11-14}
		$\nu$ & $n$ & {$c=0.1$} & {$c=0.25$} & {$c=0.5$} & {$c=1$} & {$c=0.5$} & {$c=1$} & {$c=1.5$} & {$c=2$} & {$c=0.05$} & {$c=0.1$} & {$c=0.25$} & {$c=0.5$} \\
		\midrule
	$\infty$ & 100   & 3.5   & 3.4   & 3.7   & 3.7   & 2.5   & 2.4   & 2.6   & 2.5   & 2.6   & 2.4   & 2.5   & 2.5 \\
	& 200   & 4.0   & 4.1   & 4.4   & 4.7   & 3.8   & 3.5   & 3.7   & 3.7   & 3.9   & 3.9   & 3.7   & 3.7 \\
	& 300   & 4.2   & 4.3   & 4.6   & 4.8   & 3.3   & 3.2   & 3.2   & 3.3   & 3.6   & 3.4   & 3.5   & 3.3 \\\cmidrule(lr){1-14}
	10    & 100   & 12.6  & 14.0  & 14.5  & 14.2  & 12.3  & 12.8  & 12.9  & 13.0  & 12.4  & 12.5  & 13.1  & 12.9 \\
	& 200   & 25.0  & 26.1  & 25.9  & 25.6  & 23.0  & 24.3  & 24.9  & 25.5  & 25.6  & 25.8  & 26.0  & 25.1 \\
	& 300   & 32.8  & 33.2  & 33.0  & 32.6  & 26.6  & 29.5  & 30.8  & 31.6  & 31.9  & 32.4  & 32.6  & 31.1 \\\cmidrule(lr){1-14}
	5     & 100   & 17.7  & 19.3  & 19.6  & 19.5  & 17.6  & 18.4  & 18.8  & 19.4  & 19.8  & 20.4  & 20.5  & 19.0 \\
	& 200   & 38.6  & 38.4  & 38.9  & 37.6  & 31.2  & 33.5  & 34.6  & 34.9  & 35.6  & 35.8  & 35.7  & 35.2 \\
	& 300   & 54.8  & 55.2  & 54.1  & 52.8  & 42.2  & 46.3  & 48.6  & 49.9  & 52.2  & 52.4  & 51.1  & 48.8 \\\cmidrule(lr){1-14}
	3     & 100   & 25.4  & 26.4  & 25.9  & 25.0  & 19.9  & 21.9  & 22.8  & 23.7  & 23.6  & 24.9  & 24.6  & 24.0 \\
	& 200   & 51.9  & 50.6  & 50.4  & 50.5  & 38.7  & 42.6  & 44.4  & 45.7  & 49.7  & 49.5  & 48.4  & 45.8 \\
	& 300   & 64.0  & 62.7  & 62.2  & 61.1  & 51.6  & 54.9  & 57.0  & 58.7  & 62.8  & 62.6  & 60.9  & 57.8 \\\cmidrule(lr){1-14}
	2     & 100   & 34.7  & 33.5  & 31.8  & 30.3  & 25.4  & 28.4  & 30.0  & 31.4  & 35.3  & 35.4  & 33.9  & 31.1 \\
	& 200   & 55.6  & 52.5  & 51.3  & 50.1  & 41.8  & 45.1  & 47.7  & 49.8  & 56.0  & 55.3  & 52.9  & 49.1 \\
	& 300   & 68.3  & 65.9  & 63.6  & 62.2  & 56.0  & 60.5  & 62.9  & 64.9  & 74.5  & 72.9  & 68.7  & 64.0 \\
		\bottomrule
	\end{tabular}
\end{table}

\begin{table}[tp]
	\centering\scriptsize
	\setlength\tabcolsep{4pt}
	\caption{Size and power results for assessing normality ($\Delta^{(G)}_{n,w}$) of the error terms appearing in Model D. The nominal size of the test is $\alpha=0.05$.}\label{tabD2}
	 \begin{tabular}{cccccccccccccccccccccccccccccccc}
		\toprule
		& & & \multicolumn{4}{c}{$w(t)=\exp({-ct^2})$} & \multicolumn{4}{c}{$w(t)=({1+t^2/c^2})^{-1}$} & \multicolumn{4}{c}{$w(t)=\exp({-c|t|})$} \\ \cmidrule(lr){4-7}\cmidrule(lr){8-11}\cmidrule(lr){12-15}
		$\eta$ & $\nu$ & $n$ & {$c=0.1$} & {$c=0.25$} & {$c=0.5$} & {$c=1$} & {$c=0.5$} & {$c=1$} & {$c=1.5$} & {$c=2$} & {$c=0.05$} & {$c=0.1$} & {$c=0.25$} & {$c=0.5$} \\
\midrule
0     & $\infty$ & 100   & 3.2   & 3.2   & 3.0   & 3.1   & 1.7   & 1.9   & 1.9   & 2.0   & 2.1   & 1.7   & 1.9   & 2.6 \\
      &       & 200   & 2.9   & 2.8   & 2.8   & 3.0   & 1.8   & 2.0   & 2.1   & 2.1   & 2.1   & 1.9   & 1.9   & 2.3 \\
      &       & 300   & 4.2   & 3.9   & 3.8   & 3.8   & 2.2   & 2.3   & 2.4   & 2.6   & 2.4   & 2.4   & 2.3   & 2.2 \\\cmidrule(lr){1-15}
0     & 5     & 100   & 5.8   & 6.6   & 6.5   & 6.7   & 5.4   & 5.4   & 5.3   & 5.5   & 5.4   & 5.4   & 5.9   & 6.0 \\
      &       & 200   & 5.8   & 6.3   & 6.5   & 7.0   & 6.3   & 6.8   & 6.7   & 6.7   & 6.8   & 6.8   & 6.4   & 6.7 \\
      &       & 300   & 6.9   & 7.7   & 8.6   & 9.2   & 8.8   & 8.6   & 8.9   & 9.1   & 9.2   & 8.4   & 9.0   & 9.7 \\\cmidrule(lr){1-15}
0     & 2.1   & 100   & 16.3  & 16.2  & 15.6  & 14.7  & 14.0  & 14.9  & 15.7  & 16.0  & 15.5  & 15.4  & 13.0  & 10.5 \\
      &       & 200   & 26.6  & 25.8  & 25.7  & 24.1  & 21.8  & 24.5  & 25.7  & 26.9  & 26.3  & 24.1  & 21.4  & 18.2 \\
      &       & 300   & 33.9  & 30.5  & 27.6  & 25.7  & 22.2  & 26.0  & 28.0  & 29.5  & 29.4  & 25.6  & 22.7  & 19.5 \\\cmidrule(lr){1-15}
100   & 2.1   & 100   & 39.3  & 37.4  & 35.7  & 33.2  & 23.3  & 28.0  & 29.9  & 31.0  & 31.2  & 25.7  & 21.4  & 17.7 \\
      &       & 200   & 64.3  & 59.9  & 56.1  & 52.4  & 40.0  & 47.2  & 51.2  & 53.8  & 52.7  & 44.6  & 36.6  & 30.8 \\
      &       & 300   & 77.9  & 72.3  & 67.4  & 62.1  & 51.6  & 59.9  & 63.9  & 66.9  & 65.5  & 56.3  & 48.9  & 40.6 \\
		\bottomrule
	\end{tabular}
\end{table}

\begin{table}[tp]
	\centering\scriptsize
	\setlength\tabcolsep{4pt}
	\caption{Size and power results for assessing symmetry ($\Delta^{(S)}_{n,w}$) of the error terms appearing in Model A. The nominal size of the test is $\alpha=0.05$.}\label{tab.symm.A}
	 \begin{tabular}{cccccccccccccccccccccccccccccccc}
		\toprule
		& & & \multicolumn{4}{c}{$w(t)=\exp({-ct^2})$} & \multicolumn{4}{c}{$w(t)=({1+t^2/c^2})^{-1}$} & \multicolumn{4}{c}{$w(t)=\exp({-c|t|})$} \\ \cmidrule(lr){4-7}\cmidrule(lr){8-11}\cmidrule(lr){12-15}
		$\eta$ & $\nu$ & $n$ & {$c=0.25$} & {$c=0.5$} & {$c=1$} & {$c=1.5$} & {$c=0.5$} & {$c=1$} & {$c=1.5$} & {$c=2$} & {$c=0.5$} & {$c=1$} & {$c=1.5$} & {$c=2$} \\
			\midrule
			0     & $\infty$ & 100   & 6.0   & 5.8   & 6.5   & 6.4   & 6.0   & 5.9   & 5.7   & 5.4   & 5.9   & 6.3   & 6.6   & 6.7 \\
			&       & 200   & 4.4   & 4.9   & 5.3   & 5.2   & 6.0   & 5.9   & 5.9   & 5.9   & 5.9   & 5.0   & 4.6   & 4.8 \\
			&       & 300   & 5.0   & 5.4   & 5.2   & 5.1   & 6.9   & 6.8   & 6.7   & 6.7   & 6.8   & 5.0   & 5.3   & 5.2 \\\cmidrule(lr){1-15}
			5     & $\infty$ & 100   & 23.0  & 23.5  & 23.5  & 22.4  & 22.0  & 21.2  & 20.8  & 20.1  & 20.8  & 22.6  & 23.4  & 24.0 \\
			&       & 200   & 38.2  & 39.5  & 39.1  & 37.7  & 35.5  & 34.9  & 34.1  & 32.8  & 33.9  & 37.1  & 39.6  & 39.2 \\
			&       & 300   & 48.8  & 50.1  & 49.1  & 47.2  & 46.6  & 45.9  & 44.8  & 44.0  & 45.5  & 48.8  & 49.5  & 49.5 \\\cmidrule(lr){1-15}
			20    & $\infty$ & 100   & 27.4  & 29.5  & 28.7  & 27.3  & 26.3  & 25.9  & 24.8  & 24.0  & 24.8  & 27.2  & 28.5  & 28.8 \\
			&       & 200   & 47.7  & 47.9  & 46.4  & 44.9  & 43.4  & 43.1  & 42.6  & 42.6  & 43.4  & 45.9  & 47.6  & 47.2 \\
			&       & 300   & 59.9  & 59.7  & 58.0  & 56.5  & 59.3  & 59.0  & 58.3  & 57.7  & 58.1  & 60.9  & 59.6  & 58.8 \\\cmidrule(lr){1-15}
			0     & 5     & 100   & 8.0   & 7.4   & 9.1   & 9.1   & 8.0   & 7.9   & 7.6   & 7.4   & 7.7   & 8.2   & 8.5   & 8.4 \\
			&       & 200   & 7.6   & 8.1   & 8.6   & 8.6   & 8.1   & 7.9   & 7.8   & 7.7   & 7.7   & 8.1   & 8.1   & 8.0 \\
			&       & 300   & 6.7   & 7.3   & 8.0   & 8.1   & 7.3   & 6.9   & 7.0   & 6.9   & 7.0   & 6.5   & 6.9   & 7.3 \\\cmidrule(lr){1-15}
			0     & 2.1   & 100   & 9.3   & 9.8   & 10.2  & 9.9   & 8.6   & 8.0   & 8.0   & 7.6   & 7.9   & 9.1   & 9.7   & 10.0 \\
			&       & 200   & 11.6  & 12.8  & 12.1  & 11.6  & 10.6  & 10.3  & 10.3  & 9.9   & 10.3  & 12.0  & 12.3  & 12.5 \\
			&       & 300   & 14.1  & 15.1  & 14.6  & 13.5  & 12.2  & 11.9  & 11.8  & 11.5  & 11.7  & 13.4  & 14.2  & 14.6 \\\cmidrule(lr){1-15}
			100   & 2.1   & 100   & 39.0  & 40.2  & 37.9  & 34.0  & 38.4  & 38.4  & 38.6  & 38.7  & 38.9  & 38.2  & 38.7  & 38.5 \\
			&       & 200   & 52.7  & 49.6  & 47.7  & 46.5  & 51.8  & 53.3  & 54.3  & 55.2  & 55.0  & 52.2  & 50.9  & 49.1 \\
			&       & 300   & 53.2  & 49.6  & 47.3  & 45.5  & 54.5  & 56.1  & 58.1  & 59.6  & 58.9  & 54.2  & 51.2  & 48.9 \\
		\bottomrule
	\end{tabular}
\end{table}

\begin{table}[tp]
	\centering\scriptsize
	\setlength\tabcolsep{4pt}
	\caption{Size and power results for assessing symmetry ($\Delta^{(S)}_{n,w}$) of the error terms appearing in Model B. The nominal size of the test is $\alpha=0.05$.}\label{tab.symm.B}
	 \begin{tabular}{cccccccccccccccccccccccccccccccc}
		\toprule
		& & \multicolumn{4}{c}{$w(t)=\exp({-ct^2})$} & \multicolumn{4}{c}{$w(t)=({1+t^2/c^2})^{-1}$} & \multicolumn{4}{c}{$w(t)=\exp({-c|t|})$} \\ \cmidrule(lr){3-6}\cmidrule(lr){7-10}\cmidrule(lr){11-14}
		$\nu$ & $n$ & {$c=0.25$} & {$c=0.5$} & {$c=1$} & {$c=1.5$} & {$c=0.5$} & {$c=1$} & {$c=1.5$} & {$c=2$} & {$c=0.5$} & {$c=1$} & {$c=1.5$} & {$c=2$} \\
\midrule
	$\infty$ & 100   & 6.2   & 6.0   & 6.5   & 6.0   & 5.3   & 5.4   & 5.4   & 5.4   & 5.5   & 6.1   & 6.2   & 6.4 \\
	& 200   & 5.3   & 5.9   & 6.0   & 5.9   & 6.3   & 6.3   & 6.1   & 6.0   & 6.2   & 6.5   & 5.6   & 5.7 \\
	& 300   & 5.2   & 5.9   & 5.9   & 5.9   & 6.7   & 6.8   & 6.9   & 6.8   & 6.9   & 5.2   & 5.7   & 5.9 \\\cmidrule(lr){1-14}
	10    & 100   & 16.8  & 19.1  & 19.1  & 18.0  & 15.6  & 15.3  & 15.0  & 14.5  & 14.8  & 16.2  & 18.0  & 18.4 \\
	& 200   & 28.7  & 29.9  & 30.4  & 29.9  & 26.4  & 26.5  & 25.9  & 25.2  & 25.9  & 29.2  & 28.9  & 30.2 \\
	& 300   & 42.2  & 43.3  & 43.4  & 42.8  & 38.0  & 37.5  & 36.2  & 35.5  & 36.6  & 41.6  & 42.6  & 43.3 \\\cmidrule(lr){1-14}
	5     & 100   & 26.9  & 28.5  & 28.7  & 26.5  & 26.0  & 25.1  & 24.1  & 23.2  & 24.4  & 26.2  & 27.4  & 27.8 \\
	& 200   & 45.6  & 46.0  & 45.3  & 44.2  & 43.3  & 42.7  & 41.9  & 41.4  & 42.4  & 44.2  & 45.6  & 45.6 \\
	& 300   & 56.6  & 56.0  & 54.9  & 53.3  & 55.1  & 55.1  & 54.3  & 53.6  & 54.3  & 56.7  & 56.0  & 55.4 \\\cmidrule(lr){1-14}
	3     & 100   & 36.5  & 37.5  & 35.9  & 33.0  & 34.9  & 34.8  & 34.2  & 34.0  & 35.0  & 35.8  & 36.3  & 35.8 \\
	& 200   & 55.9  & 55.8  & 54.5  & 53.0  & 55.4  & 55.5  & 55.8  & 55.5  & 55.9  & 55.8  & 55.7  & 55.2 \\
	& 300   & 66.1  & 64.6  & 62.4  & 60.1  & 63.9  & 64.1  & 63.7  & 63.5  & 63.8  & 65.7  & 64.9  & 63.5 \\\cmidrule(lr){1-14}
	2     & 100   & 41.0  & 41.1  & 39.9  & 35.7  & 38.5  & 38.3  & 37.8  & 37.5  & 37.5  & 39.1  & 40.9  & 40.4 \\
	& 200   & 60.3  & 58.6  & 55.9  & 53.4  & 58.4  & 59.2  & 59.3  & 59.4  & 59.8  & 60.1  & 59.3  & 57.9 \\
	& 300   & 66.9  & 64.6  & 61.1  & 59.4  & 66.7  & 67.0  & 67.5  & 68.4  & 68.3  & 67.2  & 65.1  & 63.8 \\
		\bottomrule
	\end{tabular}
\end{table}

\begin{table}[tp]
	\centering\scriptsize
	\setlength\tabcolsep{4pt}
	\caption{Size and power results for assessing symmetry ($\Delta^{(S)}_{n,w}$) of the error terms appearing in Model C. The nominal size of the test is $\alpha=0.05$.}\label{tab.symm.C}
	 \begin{tabular}{cccccccccccccccccccccccccccccccc}
		\toprule
		& & \multicolumn{4}{c}{$w(t)=\exp({-ct^2})$} & \multicolumn{4}{c}{$w(t)=({1+t^2/c^2})^{-1}$} & \multicolumn{4}{c}{$w(t)=\exp({-c|t|})$} \\ \cmidrule(lr){3-6}\cmidrule(lr){7-10}\cmidrule(lr){11-14}
		$\kappa$ & $n$ & {$c=0.25$} & {$c=0.5$} & {$c=1$} & {$c=1.5$} & {$c=0.5$} & {$c=1$} & {$c=1.5$} & {$c=2$} & {$c=0.5$} & {$c=1$} & {$c=1.5$} & {$c=2$} \\
\midrule
1 & 100 & 7.2 & 7.6 & 8.7 & 8.8 & 8.8 & 8.8 & 8.5 & 8.4 & 8.6 & 9.2 & 7.5 & 8.1 \\
& 200 & 7.2 & 7.6 & 7.7 & 8.8 & 8.3 & 8.0 & 7.9 & 7.7 & 8.2 & 7.3 & 7.5 & 7.5 \\
& 300 & 5.6 & 6.9 & 7.6 & 7.9 & 7.6 & 7.7 & 7.7 & 7.4 & 7.5 & 5.4 & 6.4 & 7.7 \\\cmidrule(lr){1-14}
2 & 100 & 25.8 & 25.8 & 24.7 & 22.8 & 23.0 & 23.2 & 22.4 & 21.8 & 22.7 & 24.4 & 25.5 & 25.4 \\
& 200 & 37.1 & 35.7 & 34.8 & 32.9 & 33.7 & 33.9 & 33.8 & 34.0 & 35.1 & 35.5 & 36.0 & 35.6 \\
& 300 & 49.8 & 46.7 & 45.5 & 42.1 & 48.8 & 49.7 & 49.0 & 49.3 & 51.3 & 50.3 & 48.1 & 47.0 \\\cmidrule(lr){1-14}
5 & 100 & 29.8 & 30.3 & 29.6 & 27.1 & 27.2 & 27.4 & 27.1 & 26.9 & 27.4 & 29.2 & 29.5 & 29.5 \\
& 200 & 45.8 & 44.6 & 42.9 & 41.9 & 43.6 & 43.5 & 43.8 & 43.9 & 44.1 & 45.1 & 45.0 & 44.2 \\
& 300 & 54.2 & 53.3 & 51.2 & 48.6 & 48.2 & 47.9 & 49.5 & 50.4 & 50.4 & 54.6 & 53.7 & 53.5 \\
		\bottomrule
	\end{tabular}
\end{table}

\begin{table}[tp]
	\centering\scriptsize
	\setlength\tabcolsep{4pt}
	\caption{Size and power results for assessing symmetry ($\Delta^{(S)}_{n,w}$) of the error terms appearing in Model D. The nominal size of the test is $\alpha=0.05$.}\label{tab.symm.D}
	 \begin{tabular}{cccccccccccccccccccccccccccccccc}
		\toprule
		& & & \multicolumn{4}{c}{$w(t)=\exp({-ct^2})$} & \multicolumn{4}{c}{$w(t)=({1+t^2/c^2})^{-1}$} & \multicolumn{4}{c}{$w(t)=\exp({-c|t|})$} \\ \cmidrule(lr){4-7}\cmidrule(lr){8-11}\cmidrule(lr){12-15}
		$\eta$ & $\nu$ & $n$ & {$c=0.25$} & {$c=0.5$} & {$c=1$} & {$c=1.5$} & {$c=0.5$} & {$c=1$} & {$c=1.5$} & {$c=2$} & {$c=0.5$} & {$c=1$} & {$c=1.5$} & {$c=2$} \\
			\midrule
0     & $\infty$ & 100   & 3.6   & 3.8   & 4.1   & 4.1   & 3.9   & 3.6   & 3.6   & 3.3   & 3.4   & 3.7   & 3.8   & 3.9 \\
      &       & 200   & 4.4   & 4.4   & 4.4   & 4.0   & 4.1   & 4.2   & 4.1   & 4.0   & 4.1   & 4.4   & 4.4   & 4.1 \\
      &       & 300   & 5.0   & 5.0   & 4.8   & 4.8   & 4.8   & 4.8   & 4.8   & 4.8   & 4.8   & 5.0   & 5.0   & 4.8 \\\cmidrule(lr){1-15}
0     & 5     & 100   & 5.9   & 6.4   & 6.9   & 6.5   & 5.8   & 5.4   & 5.2   & 5.1   & 5.1   & 5.8   & 6.4   & 6.7 \\
      &       & 200   & 5.9   & 6.1   & 5.9   & 5.7   & 5.7   & 5.6   & 5.6   & 5.5   & 5.6   & 5.8   & 6.0   & 6.1 \\
      &       & 300   & 5.8   & 6.0   & 5.9   & 5.6   & 5.3   & 5.3   & 5.2   & 5.0   & 5.2   & 5.6   & 5.7   & 5.9 \\\cmidrule(lr){1-15}
0     & 2.1   & 100   & 8.2   & 8.1   & 7.7   & 7.7   & 7.8   & 7.7   & 7.6   & 7.5   & 7.7   & 8.0   & 8.1   & 7.9 \\
      &       & 200   & 8.0   & 8.0   & 7.3   & 7.1   & 7.3   & 7.2   & 7.1   & 6.7   & 7.0   & 7.9   & 7.9   & 7.6 \\
      &       & 300   & 8.8   & 9.0   & 8.6   & 7.8   & 7.7   & 8.0   & 7.8   & 7.7   & 8.0   & 8.1   & 8.4   & 8.5 \\\cmidrule(lr){1-15}
100   & 2.1   & 100   & 40.3  & 39.0  & 36.9  & 35.0  & 39.6  & 40.1  & 40.1  & 41.3  & 41.0  & 39.8  & 39.1  & 38.7 \\
      &       & 200   & 53.4  & 50.4  & 47.2  & 45.4  & 54.7  & 56.9  & 59.1  & 60.1  & 59.8  & 53.8  & 51.0  & 49.7 \\
      &       & 300   & 56.6  & 52.7  & 49.0  & 46.7  & 58.6  & 61.2  & 63.7  & 65.3  & 64.9  & 57.4  & 53.8  & 51.4 \\
		\bottomrule
	\end{tabular}
\end{table}								

\section{Illustrative applications}\label{sec:applications}
For our first application of the described procedures we consider the ultrasonic
calibration data given in \emph{NIST/SEMATECH e-Handbook of Statistical Methods}
(the data can be downloaded from \url{http://www.itl.nist.gov/div898/handbook/pmd/section6/pmd631.htm}).
The response variable $Y$ represents ultrasonic response and the predictor variable $X$ is metal distance.

We investigate the appropriateness of four alternative models:
a homoskedastic model with or without transformation of the
response variable and a heteroskedastic model with or without transformation of
the response variable. For each of these models we test for validity, i.e. independence of the error term and the regressor. We employ all tests considered in this paper
(and their homoskedastic counterparts introduced by
\citealp{neumeyer2016}, and \citealp{huskova2018}),
and for all tests based on the characteristic function
we choose a Gaussian characteristic kernel.
The choice of the tuning parameter was based on the Monte Carlo study
and is shown in Table~\ref{tab:application.ultrasonic2} along with the numerical results. For this application we used 1\,000 bootstrap replications and assume a significance level of 0.05. For simplicity we used Fan and Gijbels (1996) for regression bandwidths and  \citet{silverman1986} for density estimation bandwidths.

For the homoscedastic case the results indicate a poor fit of the respective non-transformation models, but implementation of the Box--Cox transformation on the response clearly improves the fit according to all tests. An enhanced fit for the after-transformation model is also illustrated by the results corresponding to the heteroscedastic case although in this case the model can not be rejected even before transformation.    

As our second application we consider the heteroscedastic location-scale model for the Canadian cross-section wage data and the Italian GDP data; these data are also discussed in \citet{racine2017a} in the context of the non-transformation model. For the Canadian wage data there are $n=205$ observations with `age' consider as predictor for `logwage'. For the Italian GDP data there are $n=1008$ observations with `year' considered as predictor for `GDP'. Our findings (see Table 13) show that for the Canadian wage data the introduction of the transformation model seems again to improve the fit according to the KS and CM tests, while the CF-based test is robust in this respect and indicates a non-fit. For the Italian data however quite the opposite holds: The KS and CM tests indicate that neither the non-transformation nor the transformation model is appropriate, while the CF-based test shows a remarkably improved fit that clearly favours the transformation model. These results are partly in line with \citet{racine2017a} as they also find an insignificant KS statistic for the Canadian wage data but at the same test reject the location-scale presumption for the Italian data. On the other hand our findings indicate that while performing a Box--Cox transformation on the response might still lead to the same conclusion, there exist cases where this transformation could enhance the fit of the underlying model.

\begin{table}[tp]
	\centering\scriptsize
	\setlength\tabcolsep{4pt}
	\caption{Estimates of the transformation parameter under the four considered models
						when applied to the ultrasonic calibration data, along with the
						$p$-values of the tests for model validity.}\label{tab:application.ultrasonic2}
	 \begin{tabular}{lccccc}
		\toprule
		&	 & \multicolumn{2}{c}{Homoskedastic case} & \multicolumn{2}{c}{Heteroskedastic case} \\ \cmidrule(lr){3-4} \cmidrule(lr){5-6}
		& Test & No transformation & Box--Cox & No transformation & Box--Cox \\
		\midrule
		Parameter estimate &  & n/a & $\hat\vartheta =  0.458$  & n/a & $\hat\vartheta = -0.436$ \\
		\midrule
		Test for validity & KS & 0.055 & 0.366 & 0.347 & 0.751 \\
		& CM & 0.002 & 0.138 & 0.363 & 0.568 \\
		& $\Delta_{n,W}$ ($c=1$) & 0.025 & 0.403 & 0.132 & 0.294 \\
		\bottomrule
	\end{tabular}
\end{table}

\begin{table}[tp]
	\centering\scriptsize
	\setlength\tabcolsep{4pt}
	\caption{Results for the Canadian cross-section wage data and Italian GDP data, along with the
						$p$-values of the tests for model validity.}\label{tab:application.canada2}
	 \begin{tabular}{lccccc}
		\toprule
		&	 &  \multicolumn{2}{c}{Canadian wage data}  & \multicolumn{2}{c}{Italian GDP data} \\ \cmidrule(lr){3-4} \cmidrule(lr){5-6}
		& Test & No transformation & Box--Cox & No transformation & Box--Cox \\
		\midrule
		Parameter estimate && n/a & $\hat\vartheta = 2.842$ & n/a & $\hat\vartheta = -0.202$ \\
		\midrule
		Test for validity & KS & 0.198 & 0.331 & <0.001 & <0.001 \\
		& CM & 0.074 & 0.128 & <0.001 & <0.001 \\
		& $\Delta_{n,W}$ ($c=1$) & 0.021 & 0.020 & <0.001 & 0.960 \\
		\bottomrule
	\end{tabular}
\end{table}

\section{Conclusions}
New tests for the validity of the heteroskedastic transformation model are proposed which are based on the well known factorization property of the joint characteristic function into its corresponding marginals. The asymptotic null distribution is derived and the consistency of the new criteria is shown. A Monte Carlo study is included by means of which a resampling version of the proposed method is compared to earlier methods and shows that the new test, aside from being computationally convenient, compares well and often outperforms its competitors, particularly under heavy tailed error distributions. A further Monte Carlo study of characteristic-function based tests for symmetry and normality of regression errors exhibit analogous favourable features. Finally a couple of illustrative applications on real data lead to interesting conclusions.  

\section*{Acknowledgements}
The work of the first author was partially supported by the grant GA\v{C}R 18-08888S.
The work of the third author was partially supported by OP RDE project No.\ CZ.02.2.69/0.0/0.0/16\_027/0008495, International Mobility of Researchers at Charles University.

\section{Appendix}
\subsection{Assumptions}

 We start with formulation of the assumptions and then we give the assertions on behavior of the test statistics under both the null hypothesis and some alternatives.

   \begin{itemize}

  \item[(A.1)] $(Y_j,\boldsymbol X_j),\,j=1,\ldots, n$, are i.i.d. random vectors, where the covariates $\boldsymbol X_j,\,j=1,\ldots, n$,
      have a compact support $\mathcal{R}_{X}$ with $\mathcal{R}_{X}\subset \mathbb{R}^p$.

      \medskip

      \item[(A.2)] We use a product kernel $K(\boldsymbol y)=\prod_{s=1}^{p} k(y_s),\quad  \boldsymbol y= (y_1,\ldots, y_{p})^\top$,  with $k(\cdot)$ being  symmetric  and continuous  in  $[-1,1]$,  and satisfying
          $$
          \int_{-1}^1 u^r k(u) du=\delta_{r,0},\quad r=0,\ldots, p,\quad \int_{-1}^1 u^{p+1} k(u) du\ne 0,
          $$
          where $\delta_{r,s}$ stands for Kronecker's delta.

           \medskip

          \item[(A.3)]  The bandwidth $h=h_n$ satisfies
          $$
          (n h^{p})^{-1} +   n h^{p +\delta}\to 0,\quad as\quad n\to \infty
          $$
           for some $\delta >0$.

             \medskip

          \item[(A.4)]  It holds that $\mathbb{E} ||\boldsymbol X_j||^2<\infty$,  and that $\boldsymbol X_j$  has the density  $f(\cdot)$ satisfying
     \begin{align*}
    0<&\inf_{{\boldsymbol{x}}\in \mathcal{R}_{X}} f(\boldsymbol x)\leq \sup_{{\boldsymbol {x}}\in \mathcal{R}_{X}} f(\boldsymbol x)<\infty,\\
     \Big| f& (\boldsymbol x)- L(f, \boldsymbol y, \boldsymbol x-\boldsymbol y, s_1)\Big|\leq
      ||\boldsymbol x-\boldsymbol y||^{s_1+\delta_1} d_1(\boldsymbol
      y),
     \end{align*}
     where $L(f, \boldsymbol y,\boldsymbol x-  \boldsymbol y, s_1)$
      is the Taylor expansion of
      the density $f$ of order $s_1$ at $\boldsymbol x$,  $\delta_1>0$ and
      $\mathbb{E}|d_1(\boldsymbol X_j)|^{2}<\infty, $ for some
      $s_1+1\geq  p/2$.

           \medskip

          \item[(A.5)]  It is assumed  $m_{\bb {\vartheta}_0}(\boldsymbol x),\,\sigma_{\bb {\vartheta}_0}(\boldsymbol x),\, \boldsymbol x \in  {\mathcal{R}}_{X}$ satisfies
           \begin{align*}
     \Big| m_{\bb {\vartheta}_0} &(\boldsymbol x)- L_{\bb {\vartheta}_0}(m, \boldsymbol
      y_0,\boldsymbol x- \boldsymbol y_0, s_2)\Big|\leq ||\boldsymbol x-\boldsymbol y_0||^{s_2+\delta_2} d_2(\boldsymbol y_0),\\
      \Big| \sigma_{\bb {\vartheta}_0} &(\boldsymbol x)- L_{\bb {\vartheta}_0}(\sigma, \boldsymbol
      y_0,\boldsymbol x- \boldsymbol y_0, s_3)\Big|\leq ||\boldsymbol x-\boldsymbol y_0||^{s_3+\delta_2} d_3(\boldsymbol y_0)
     \end{align*}
     where $L_{\bb {\vartheta}_0}(m_{\bb {\vartheta}_0}, \boldsymbol y_0, \boldsymbol x- \boldsymbol y_0, s_2)$
      is the Taylor expansions of regression function $m_{\bb {\vartheta}_0}$ of order $s_2$ at $\boldsymbol y_0$ and
      $\mathbb{E}|d_2(\boldsymbol X_{j})|^{2}<\infty$ for some
        $s_2 \geq  p/2$  and $\mathbb{E} m_{\bb {\vartheta}_0}^2 (\boldsymbol X_{j})<\infty$. Similarly,
        $L_{\bb {\vartheta}_0}(\sigma_{\bb {\vartheta}_0}, \boldsymbol y_0, \boldsymbol x- \boldsymbol y_0, s_3)$
      is the Taylor expansions of regression function $\sigma_{\bb {\vartheta}_0}$ of order $s_3$ at $\boldsymbol y_0$ and
      $\mathbb{E}|d_3(\boldsymbol X_{j})|^{2}<\infty$ for some
        $s_3 \geq  p/2$  and $\mathbb{E} \sigma_{\bb {\vartheta}_0}^2 (\boldsymbol X_{j})<\infty$.

       \medskip

          \item[(A.6)] $\mathcal{L}=\{T_{\boldsymbol { \vartheta}};\, \boldsymbol { \vartheta}\in \Theta\}$
          is a parametric class of  strictly increasing transformations,  $\Theta$ is a open  measurable subset of
          $\mathbb{R}^q$, and for some $\xi>0$,
                  \begin{align*}
     \sup_{||\bb {\vartheta}-\bb {\vartheta}_0|| \leq \xi}\Big|{\mathcal{T}}_{\bb {\vartheta}}(Y_j)&- {\mathcal{T}}_{\bb {\vartheta}_0} (Y_j)
      -\sum_{s=1}^{q} (\vartheta_s - \vartheta_{s0}) \frac{\partial  {\mathcal{T}}_{\bb {\vartheta}}(Y_j) }
      {\partial
      \vartheta_s}\Big|_{\bb {\vartheta}=\bb{\vartheta}_0}\Big|
     /|| \bb{\vartheta}-\bb{\vartheta}_0 ||^{1+\delta_4}
     \\&\leq
d_4(Y_j),
     \end{align*}
where $\mathbb{E}d^2_4(Y_j)<\infty$,   $  \mathbb{E}(|d_4(Y_j)||\boldsymbol X_j)<\infty$,
a.s., and for some  $\delta_4>0$ it holds that
\begin{align*}
 \mathbb{E}&\Big( \frac{\partial  {\mathcal{T}}_{\bb{\vartheta}}(Y_j) }
      {\partial
      \vartheta_s}\Big|_{\bb{\vartheta}=\bb{\vartheta}_0}\Big|\boldsymbol X_{j}\Big)= \frac{\partial \mathbb{E}\big( {\mathcal{T}}_{\bb{\vartheta}}(Y_j)|\boldsymbol X_{j}\big) }
      {\partial
      \vartheta_s}\Big|_{\bb{\vartheta}=\bb{\vartheta}_0},
      \text{ a.s.},
        \\
        \mathbb{E}&\Big( \mathbb{E}\Big( \frac{\partial  {\mathcal{T}}_{\bb{\vartheta}}(Y_j) }
      {\partial
      \vartheta_s}\Big|_{\bb{\vartheta}=\bb{\vartheta}_0}\Big|\boldsymbol
      X_{j}\Big)\Big)^2<\infty,\quad   \mathbb{E} {\mathcal{T}}_{\bb{\vartheta}}^2(Y_j)<\infty.
\end{align*}

\medskip

              \item[(A.7)]  The estimator $\widehat {\boldsymbol { \vartheta}}$ of $\boldsymbol { \vartheta}_0$ (they are $q$-dimensional) satisfies:
               $$ \sqrt n\Big(\widehat {\boldsymbol { \vartheta}}-\boldsymbol { \vartheta}_0\Big)=
               \frac{1}{\sqrt n} \sum_{j=1}^n \boldsymbol g(Y_j, \boldsymbol X_j)+o_P(1)
                  $$
                  where $ \boldsymbol g(Y_j, \boldsymbol X_j)$  is  has zero mean and finite covariance matrix.

\medskip
 \item[(A.8)]  The weight   function is such that $W(t_1,\bb t_2)=w_1(t_1)w_2(\bb t_2)$, where $w_m(\cdot)$ satisfy
                  \begin{align*}
   w_1(t)=w_1(-t), \ t\in\mathbb R,  \int_{-\infty}^\infty & t^2 w_1(t)dt<\infty, \ \ w_2({\boldsymbol{t}})=w_2(-{\boldsymbol{t}}), \boldsymbol{t} \in \mathbb R^p.  
\end{align*}

  \medskip
 \item[(A.9)]  $\mathcal{L}=\{T_{\boldsymbol { \vartheta}};\, \boldsymbol { \vartheta}\in \Theta\}$
          is a parametric class of  strictly increasing transformations,  $\Theta$ is a open  measurable subset of
          $\mathbb{R}^q$ and that for all  $\bm{\vartheta} \in \Theta$
 $$ \Big|{\mathcal{T}}_{\bb {\vartheta}}(Y_j)- {\mathcal{T}}_{\bb {\vartheta}^0} (Y_j)|\leq  d_5(Y_j)||\bb {\vartheta}-\bb {\vartheta}^0||,
$$
with $E|d_5(Y_j)|<\infty$.

  \end{itemize}

\medskip

\noindent{Comments on the assumptions}:
\begin{itemize}

\item Assumptions (A.2) are  (A.3) are quite standard.
\item  Assumption (A.4) requires smoothness of the density
 $f(\cdot)$ of $\boldsymbol X$.
\item Assumption (A.5) formulates the requirements on the  regression function
 $m_{\boldsymbol{\vartheta}_0}( \boldsymbol x)= \mathbb{E} \big({\mathcal{T}}
 _{\bb{\vartheta}} (Y_j)| \boldsymbol X_{j}=\boldsymbol x\big)$.
  Motivation for   assumptions (A.4) and (A.5) are  from \citet{delgado2001}.
\item Assumption (A.7) requires
  that a $\sqrt n$-estimator of  $\boldsymbol { \vartheta}_0$   with an asymptotic representation is available.
	Such estimators are proposed and studied in, e.g., \citet{breiman1985}, \citet{horowitz2009} and \citet{linton2008}.  They  are
either based on a  modified least squares method or on profile likelihood  estimators
 or  on mean square distance from independence.

\item  Assumptions (A.9) and (A.10) are for the considered class of alternatives.

    \end{itemize}

\subsection{ Proofs}

The proofs are quite technical and therefore we present the main steps only. Additionally,
main line of the proofs follows that in \citet{hlavka2011},
however some modifications  and extensions are needed.  Standard technique from nonparametric
regression is applied together with functional central limit theorems.

When  not confused we use short notations:
$$\varepsilon_j={\varepsilon}_{\vartheta_0,j},\quad \widehat{\varepsilon}_j=\widehat{\varepsilon}_{\widehat\vartheta,j}, \quad
\widehat m_n (\boldsymbol x)= \widehat m_{\widehat\vartheta} (\boldsymbol x),\quad
\widehat \sigma_n (\boldsymbol x)= \widehat \sigma_{\widehat\vartheta} (\boldsymbol x).
$$
For simplicity we give the proofs only for $\vartheta$ univariate. Multivariate situation proceeds quite analogously.

\begin{proof}[Proof of Theorem~\ref{thm1}]
By assumption (A.7) and elementary
properties of the functions sinus and cosinus
  \begin{equation} \label{equiv}
  \Delta_{n,W}= \int\int \left|J_{n,1} (t_1,\boldsymbol t_2)+ J_{n,2} (t_1,\boldsymbol t_2)
   \right|^2 W(t_1,\boldsymbol t_2) dt_1 d\boldsymbol t_2,
  \end{equation}
where
\begin{align*}
J_{n,1} (t_1, \boldsymbol  t_2)={}&{}\frac{1}{\sqrt n} \sum_{j=1}^n \Big[\big(\cos ( t_1
 \widehat{\varepsilon}_j)-E\cos ( t_1
 {\varepsilon}_j)\big)
 g_+ (\boldsymbol t_2^\top \boldsymbol X_j)\\
 &{}+\big(\sin ( t_1 \widehat
 {\varepsilon}_j)-E\sin ( t_1
 {\varepsilon}_j)\big)
 g_-(\boldsymbol t_2^\top \boldsymbol  X_j) \Big],\\
J_{n,2} (t_1,\boldsymbol t_2)={}&{}
  \frac{1}{n^{3/2}} \sum_{j=1}^n \sum_{v=1}^n \Big[\big(\cos(t_1\widehat {\varepsilon}_v)-E\cos ( t_1
 {\varepsilon}_v)\big)g_+(\boldsymbol  t_2^\top \boldsymbol  X_j)
  \\
   &{}+ \big(\sin(t_1\widehat \varepsilon_v)-E\sin ( t_1
 {\varepsilon}_v)\big) g_-(\boldsymbol t_2^\top\boldsymbol   X_j)
  \Big]
\end{align*}
with
\begin{align*}
g_+(\boldsymbol t_2^\top \boldsymbol X)=&\cos (\boldsymbol t_2^\top \boldsymbol X)+ \sin(\boldsymbol t_2^\top \boldsymbol  X)- E\big(\cos (\boldsymbol t_2^\top \boldsymbol X)+ \sin(\boldsymbol t_2^\top \boldsymbol  X)\big),\\
g_-(\boldsymbol t_2^\top \boldsymbol X)=&\cos (\boldsymbol t_2^\top \boldsymbol X)- \sin(\boldsymbol t_2^\top \boldsymbol  X)- E\big(\cos (\boldsymbol t_2^\top \boldsymbol X)- \sin(\boldsymbol t_2^\top \boldsymbol  X)\big).
\end{align*}
 A useful asymptotic representation for $J_{n,1} (t_1,\boldsymbol  t_2)$
 provides \ref{lem1}
   while negligibility of $J_{n,2} (t_1, \boldsymbol t_2)$ is proved quite analogously and therefore its proof is omitted.
\end{proof}

\begin{lem}\label{lem1}
 Let the assumptions of Theorem~\ref{thm1} be satisfied then, as $n\to
 \infty$,
\[
\int\int \Big| J_{n,1} (t_1,\boldsymbol t_2)- Q_{\varepsilon,\boldsymbol X,c}(t_1,\boldsymbol  t_2)- Q_{\varepsilon,\boldsymbol X,s}(t_1,\boldsymbol t_2)-L_{\varepsilon,\boldsymbol X}(t_1,\boldsymbol t_2)
\Big|^2 W(t_1,\boldsymbol t_2) dt_1 d \boldsymbol t_2\to^P 0,
\]
     where
     \begin{align*}
  Q_{\varepsilon,\boldsymbol X, c}(t_1, \boldsymbol  t_2)&= \frac{1}{\sqrt n}\sum_{j=1}^n\big\{( \cos(t_1 \varepsilon_j) - E\big(\cos(t_1 \varepsilon_j))) + t_1\varepsilon_jS_{\varepsilon}(t_1) \\
  &\qquad\qquad\qquad + t_1(\varepsilon_j^2-1)C_{\varepsilon}'(t_1)/2\big\} g_+(\boldsymbol  t_2^\top \boldsymbol X_j) ,
 \\
  Q_{\varepsilon,\boldsymbol  X, s}(t_1,\boldsymbol t_2)&= \frac{1}{\sqrt n}\sum_{j=1}^n\big\{(\sin(t_1 \varepsilon_j)- E\big(\sin(t_1 \varepsilon_j)\big))- t_1\varepsilon_jS_{\varepsilon}(t_1) \\
  &\qquad\qquad\qquad
 - t_1 (\varepsilon_j^2-1)C_{\varepsilon}'(t_1)/2\big\} g_-(\boldsymbol  t_2^\top \boldsymbol  X_j)
 ,\\
 L_{\varepsilon,\boldsymbol X}(t_1,\boldsymbol t_2)&= \sqrt n(\widehat \vartheta-\vartheta_0)
  H_{\vartheta_0,1}(t_1,\boldsymbol t_2)
 \end{align*}
where $C_{\varepsilon}$ and $S_{\varepsilon}$ are the real and the imaginary part of the CF of $\varepsilon_j$ and $C'_{\varepsilon}$ and
$S'_{\varepsilon}$ are respective derivatives  and  $ H_{\vartheta_0,1}(t_1,\boldsymbol t_2) $ is defined in Theorem 1 with $q=1$.
\end{lem}

\begin{proof}
Recall that the residuals
$\widehat \varepsilon_j$
  can be expressed as
 \begin{equation}\label{resid2}
 \widehat \varepsilon_j= \varepsilon_{j} +\varepsilon_j\left\{ \frac{\sigma_{\vartheta_0} (X_j)}{\widehat
 \sigma_n(X_j)}-1\right\} +\frac{m_{\vartheta_0}(X_j)-\widehat m_n(X_j)}{
 \widehat\sigma_n(X_j)}+\frac{{\mathcal{T}}_{\bb {\vartheta}}(Y_j)-{\mathcal{T}}_{\bb {\vartheta}_0}(Y_j)}{\widehat\sigma_n(X_j)},\, j=1,\ldots,n,
\end{equation}
 and then by Taylor expansion  and smoothness of ${\mathcal{T}}_{\vartheta}(Y_j) $ w.r.t. $\vartheta$ we have
 \begin{equation}\begin{split}\label{taylor}
  {}&{} \cos (t_1\widehat \varepsilon_j) \\
 {}&{}= \cos (t_1\varepsilon_j) \\
 &{}\quad- t_1\sin (t_1\varepsilon_j) \Big[ \varepsilon_j \left\{ \frac{\sigma_{\vartheta_0} (X_j)}{\widehat
	 \sigma_n(X_j)}-1\right\} +\frac{m_{\vartheta_0} (X_j)-\widehat m_n(X_j)}{
	 \widehat\sigma_n(X_j)} +\frac{{\mathcal{T}}_{\widehat\vartheta}(Y_j)-{\mathcal{T}}_{\vartheta_0}(Y_j)}
	 {\widehat\sigma_n(X_j)} \Big] \\
 &{}\quad+ t_1^2 R^c_{nj}(t_1),
  \end{split}\end{equation}
  $j=1,\ldots,n$, where  $R^c_{nj}(t)$'s are remainders.
  Similar relations can be obtained for the $\sin (t\widehat
  \varepsilon_j)$.
   Comparing the present  situation with that considered in \citet{hlavka2011} there is an additional parameter $\theta$ which influences the behavior of the treated variables.  We  notice
  \begin{align*}
   &{}\widehat m_{ \vartheta}(\boldsymbol X_j)-m_{\vartheta_0}(\boldsymbol X_j)\\ ={}&{}
   \frac{1}{ n h^p \widehat f(\bb X_j)}\sum_{v=1}^n K\Big(\frac{\bb X_v-\bb X_j}{h}\Big)
   \Big({\mathcal{T}}_{\vartheta}(Y_v) -{\mathcal{T}}_{\vartheta_0}(Y_v)\Big) \\
   &{}+ \frac{1}{ n h^p\widehat f(\bb X_j)}\sum_{v=1}^n K\Big(\frac{\bb X_v-\bb X_j}{h}\Big)\Big(\varepsilon_{v}\sigma_{\vartheta_0}
(\bb X_v) + m_{\vartheta_0}(\boldsymbol X_v) -m_{\vartheta_0}(\boldsymbol X_j)\Big)\\
={}&{} A_{1,n}(\bb X_j,\vartheta)+  A_{2,n}(\bb X_j, \vartheta),say
\end{align*}
and
\begin{align*}
&\widehat \sigma_{\vartheta}^2 (\bb X_j) \\
{}&{}=\frac{1}{ n h^p \widehat f(\bb X_j)}\sum_{v=1}^n K\Big(\frac{\bb X_v-\bb X_j}{h}\Big)
\Big({\mathcal{T}}_{\vartheta}(Y_v) -{\mathcal{T}}_{\vartheta_0}(Y_v)
 +\varepsilon_{v}\sigma_{\vartheta_0}(\bb X_j) \\
 &{}\quad
 + m_{\vartheta_0}(\boldsymbol X_v) -m_{\vartheta_0}(\boldsymbol X_j)\Big)^2\\
 {}&{}=\frac{1}{ n h^p\widehat f(\bb X_j)}\sum_{v=1}^n K\Big(\frac{\bb X_v-\bb X_j}{h}\Big)
\Big({\mathcal{T}}_{\vartheta}(Y_v) -{\mathcal{T}}_{\vartheta_0}(Y_v)\Big)^2\\
&\quad+
 \frac{1}{ n h^p  \widehat f(\bb X_j)}\sum_{v=1}^n K\Big(\frac{\bb X_v-\bb X_j}{h}\Big)\Big(\varepsilon_{v}\sigma_{\vartheta_0}(\bb X_j)
 + m_{\vartheta_0}(\boldsymbol X_v) -m_{\vartheta_0}(\boldsymbol X_j)\Big)^2\\
 &\quad+2 \frac{1}{ n h^p \widehat f(\bb X_j)}\sum_{v=1}^n K\Big(\frac{\bb X_v-\bb X_j}{h}\Big)
\Big({\mathcal{T}}_{\vartheta}(Y_v) -{\mathcal{T}}_{\vartheta_0}(Y_v)\Big)
 \varepsilon_{v}\sigma_{\vartheta_0}(\bb X_v)\\
  &\quad+2 \frac{1}{ n h^p \widehat f(\bb X_j)}\sum_{v=1}^n K\Big(\frac{\bb X_v-\bb X_j}{h}\Big)
\Big({\mathcal{T}}_{\vartheta}(Y_v) -{\mathcal{T}}_{\vartheta_0}(Y_v)\Big)
\Big( m_{\vartheta_0}(\boldsymbol X_v) -m_{\vartheta_0}(\boldsymbol X_j)\Big)\\
={}& B_{1,n}(\bb X_j,\vartheta)+  B_{2,n}(\bb X_j, \vartheta)+  B_{3,n}(\bb X_j, \vartheta) +B_{4,n}(\bb X_j, \vartheta),say
 \end{align*}
  The terms  $A_{1,n}(\bb X_j,\vartheta), A_{2n}(\bb X_j, \vartheta),  B_{2,n} (\bb X_j, \vartheta), B_{3,n}(\bb X_j,\vartheta)$ are influential while the others are properly negligible. The desired  property of $ A_{2n}(\bb X_j, \vartheta),  B_{2,n} (\bb X_j, \vartheta)$ follow from  results in \citet{hlavka2011} and  are formulated below, the remaining   terms will be discussed  shortly later.

 Proceeding as in \citet{hlavka2011} we find out that
 \begin{multline*}
 \int \int \Big( Q_{\varepsilon,\boldsymbol X, c}(t_1, \boldsymbol  t_2) -
  \frac{1}{\sqrt n}\sum_{j=1}^ng_+(\bb t_2 \bb X_j)\Big\{( \cos(t_1 \varepsilon_j) - E\big(\cos(t_1 \varepsilon_j)\big)\\
  -t_1\sin(t_1\varepsilon_j)\Big[\frac{- A_{2,n}(\bb X_j,\vartheta)}{\sigma_{\vartheta_0} (X_j)} - \frac{\varepsilon_j}{2}\Big(\frac{B_{2,n}(\bb X_j,\vartheta)}{\sigma^2_{\vartheta_0}(X_j)}-1\Big)\Big]\Big \}  \Big)^2 W(t_1,\boldsymbol t_2) dt_1 d \boldsymbol t_2\to^P 0,
  \end{multline*}
  Next we study  the remaining influential terms, i.e., terms that depends on $\vartheta$. We start with
  \begin{align*}
  \frac{1}{\sqrt n}&\sum_{j=1}^n g_+(\bb t_2 \bb X_j)
  \Big(-t_1\sin(t_1\varepsilon_j)\frac{- A_{1,n}(\bb X_j,\vartheta)-\varepsilon_j B_{3,n}(\bb X_j,\vartheta) + {\mathcal{T}}_{\vartheta}(Y_j) -{\mathcal{T}}_{\vartheta_0}(Y_j)}{\sigma_{\vartheta_0}(\bb X_j)}\Big)\\
  =\frac{1}{\sqrt n}&\sum_{j=1}^n
 g_+(\bb t_2 \bb X_j) (-t_1\sin(t_1\varepsilon_j))\frac{1}{\sigma_{\vartheta_0}(\bb X_j)}\Big[
  {\mathcal{T}}_{\vartheta}(Y_j) -{\mathcal{T}}_{\vartheta_0}(Y_j) \\&-\frac{1}{nh^p \widehat f(\bb X_j)}\sum_{v=1}^n K\Big(\frac{\bb X_v-\bb X_j}{h}\Big) \big( {\mathcal{T}}_{\vartheta}(Y_v) -{\mathcal{T}}_{\vartheta_0}(Y_v)\big) \big(1+ \varepsilon_{v}  \varepsilon_{j}\big)
\Big] \\
=& \frac{1}{\sqrt n}\sum_{j=1}^n \sum_{v=1}^n K\Big(\frac{\bb X_v-\bb X_j}{h}\Big)
\Big( {\mathcal{T}}_{\vartheta}(Y_j) -{\mathcal{T}}_{\vartheta_0}(Y_j)\Big) \\&\times\Big[\frac{1}{nh^p \widehat f(\bb X_j) \sigma_{\vartheta_0}(\bb X_j)}
(g_+(\bb t_2 \bb X_j) (-t_1\sin(t_1\varepsilon_j))
\\ &- \frac{1}{nh^p \widehat f(\bb X_v) \sigma_{\vartheta_0}(\bb X_v)}g_+(\bb t_2 \bb X_v))
  (-t_1\sin(t_1\varepsilon_v))(1+ \varepsilon_{j}\varepsilon_v)\Big].
 \end{align*}

  This behaves  asymptotically  as (due to assumption (A.7))
  \begin{align*}
  \sqrt n& (\widehat\vartheta-\vartheta_0)\frac{1}{ n}\sum_{j=1}^n \sum_{v=1}^n K\Big(\frac{\bb X_v-\bb X_j}{h}\Big)
 \frac{\partial{\mathcal{T}}_{\vartheta}(Y_j)}{\partial  \vartheta}\Big|_{\vartheta=\vartheta_0}  \\&\times\Big[\frac{1}{nh^p \widehat f(\bb X_j) \sigma_{\vartheta_0}(\bb X_j)}
(g_+(\bb t_2 \bb X_j) (-t_1\sin(t_1\varepsilon_j))
\\ &- \frac{1}{nh^p \widehat f(\bb X_v) \sigma_{\vartheta_0}(\bb X_v)}g_+(\bb t_2 \bb X_v))
  (-t_1\sin(t_1\varepsilon_v))(1+\varepsilon_{j\theta_0}\varepsilon_v\Big]
  \end{align*}

   Since by the assumptions $\sqrt n (\widehat\vartheta-\vartheta_0)=O_P(1)$ it suffices to to study
   \begin{align*}
   C_{1,n}(t_1,\bb  t_2)=&\frac{1}{ n}\sum_{j=1}^n
 \frac{\partial{\mathcal{T}}_{\vartheta}(Y_j)}{\partial  \vartheta}\Big|_{\vartheta=\vartheta_0}  \frac{1}{ \sigma_{\vartheta_0}(\bb X_j)}
g_+(\bb t_2 \bb X_j) (-t_1\sin(t_1\varepsilon_j))
\\
   C_{2,n}(t_1,\bb t_2)=&-\frac{1}{ n}\sum_{j=1}^n \sum_{v=1}^n K\Big(\frac{\bb X_v-\bb X_j}{h}\Big)
 \frac{\partial{\mathcal{T}}_{\vartheta}(Y_j)}{\partial  \vartheta}\Big|_{\vartheta=\vartheta_0}   \frac{1}{nh^p \widehat f(\bb X_v) \sigma_{\vartheta_0}(\bb X_v)}g_+(\bb t_2 \bb X_v))
 \\&\times (-t_1\sin(t_1\varepsilon_v))(1+\varepsilon_{j}\varepsilon_v)
     \end{align*}

By the law of large numbers (uniform in $t_1,\bb t_2$ ), as $n\to \infty$,
\begin{align*}
C_{1,n}(t_1,\bb  t_2)& \to^P E\Big(\frac{\partial{\mathcal{T}}_{\vartheta}(Y_1)}{\partial  \vartheta}\Big|_{\vartheta=\vartheta_0}  \frac{1}{ \sigma_{\vartheta_0}(\bb X_1)}
g_+(\bb t_2 \bb X_1) (-t_1\sin(t_1\varepsilon_1))\Big),\\
 C_{2,n}(t_1,\bb t_2)&\to^P -
 E\Big(\frac{\partial{\mathcal{T}}_{\vartheta}(Y_1)}{\partial  \vartheta}\Big|_{\vartheta=\vartheta_0}
 \frac{1}{\sigma_{\vartheta_0} (\boldsymbol  X_2)} (1+\varepsilon_1 \varepsilon_2)\Big(-t_1\sin(t_1\varepsilon_2)g_+(\boldsymbol t_2^\top\boldsymbol X_2)
 \Big)\Big).
 \end{align*}
  Similarly we proceed with
  $$
   \frac{1}{\sqrt n}\sum_{j=1}^n g_-(\bb t_2 \bb X_j)
  \Big(t\cos(t\varepsilon_j)\frac{- A_{1,n}(\bb X_j,\vartheta)-\varepsilon_j B_{3,n}(\bb X_j,\vartheta) + {\mathcal{T}}_{\vartheta}(Y_j) -{\mathcal{T}}_{\vartheta_0}(Y_j)}{\sigma_{\vartheta_0}(\bb X_j)}\Big)
  $$
 and the proof of Lemma \ref{lem1} is finished.
\end{proof}

\begin{proof}[Proof of Theorem~\ref{thm2}] We proceed as in the proof of Theorem~\ref{thm1} and~Lemma \ref{lem1} and come to the conclusion that
\[
\frac{1}{n}\Delta_{n,W} \to^P  \int_{ \mathbb R^p} | \varphi_{\boldsymbol X,\varepsilon_{\bb{\vartheta}^0}} (\boldsymbol t_2,t_1)-\varphi_{\boldsymbol X} (\boldsymbol t_2)  \varphi_{\varepsilon_{\bb{\vartheta}^0}} ( t_1)
|^2
W(t_1,\boldsymbol t_2) dt_1d\boldsymbol t_2>0. \qedhere
\]
\end{proof}

\end{document}